\documentclass[11pt]{article}

\usepackage{color}

\usepackage{amsmath,amsthm}
\usepackage{amsfonts}
\usepackage{latexsym}
\usepackage{indentfirst}
\usepackage{psfrag,epsf}
\usepackage{pstool}

\usepackage{epsfig, subfigure}
\usepackage{graphicx}
\usepackage{epstopdf}
\usepackage{enumerate}
\usepackage[subnum]{cases}
\usepackage{geometry}
\usepackage{fullpage}
\usepackage{tikz}

\usetikzlibrary{shapes,arrows,calc}
\usepackage{multicol}  

\usepackage{bm}

\newtheorem{theorem}{Theorem} \rm
\newtheorem{lemma}[theorem]{Lemma}

\newtheorem{claim}[theorem]{Claim}

\theoremstyle{plain}

\title{The Alon--Tarsi Number of Squares of Subcubic Planar Graphs without Cycles of Lengths $4$ to $8$}\author
{
Seog-Jin Kim$^{\rm a, b}$\thanks{E-mail: {\tt skim12@konkuk.ac.kr
}},
Xiaopan Lian$^{\rm c}$\thanks{E-mail: {\tt Lian@nankai.edu.cn
}},
Rong Luo$^{\rm d}$\thanks{E-mail: {\tt rluo@math.wvu.edu
}},
\\
{\footnotesize$^{\rm a}$ Department of Mathematics Education, Konkuk University, Seoul, Korea}\\
{\footnotesize$^{\rm b}$ Korea Institute for Advanced Study (KIAS), Seoul, Korea.}\\
{\footnotesize$^{\rm c}$ Center for Combinatorics and LPMC, Nankai University, China}\\
{\footnotesize$^{\rm d}$ Department of Mathematics, West Virginia University, USA}\\
}
\date{}
\begin{document}

\maketitle

\begin{abstract}

The Alon--Tarsi number $AT(G)$ of a graph $G$, defined via the graph polynomial, is a strengthening of the list chromatic number $\chi_{\ell}(G)$. We study the Alon--Tarsi number of squares of planar graphs. The \emph{square} of a graph $G$ is the graph obtained by joining every pair of vertices whose distance in $G$ is at most $2$. Recently, Kim and Luo (2026) proved that $\chi_{\ell}(G^2)\le 6$ for every subcubic planar graph containing no $k$-cycles for $4\le k\le 8$. We strengthen this result by proving that $AT(G^2)\le 6$ for every such graph $G$.

\medskip


\end{abstract}

\section{Introduction}

The Alon--Tarsi method plays a fundamental role in graph coloring. It provides an algebraic approach to bounding the list chromatic number by identifying monomials with nonzero coefficients in the graph polynomial. More importantly, the resulting parameter, called the \emph{Alon--Tarsi number}, is stronger than the list chromatic number. Consequently, an upper bound on the Alon--Tarsi number not only yields the same bound on the list chromatic number but also demonstrates that the bound can be established using algebraic techniques.

We begin by introducing the necessary definitions. Given a graph $G$, a \emph{list assignment} $L$ assigns to each vertex a list of colors. The graph $G$ is \emph{$L$-colorable} if it admits a proper coloring $f$ such that $f(v)\in L(v)$ for every vertex $v$. If $G$ is $L$-colorable whenever every list has size at least $k$, then $G$ is said to be \emph{$k$-choosable}. The \emph{list chromatic number} $\chi_{\ell}(G)$ is the smallest integer $k$ for which $G$ is $k$-choosable.

A stronger notion than list coloring is Alon--Tarsi coloring, introduced through the graph polynomial in \cite{JT95}. Let $G$ be a graph, and fix an arbitrary ordering `$<$' of its vertices. The \emph{graph polynomial} of $G$ is defined by
\[
P_G(\bm{x})=\prod_{u\sim v,\;u<v}(x_u-x_v),
\]
where $u\sim v$ indicates that $u$ and $v$ are adjacent, and $\bm{x}=(x_v)_{v\in V(G)}$ is the vector of variables indexed by the vertices of $G$. Observe that $P_G(\bm{x})$ is a homogeneous polynomial of degree $|E(G)|$.

Let $f:V(G)\rightarrow\{0,1,2,\ldots\}$ be a function. We say that $G$ is \emph{Alon--Tarsi $f$-choosable} if $P_G(\bm{x})$ contains a monomial
\[
\prod_{v\in V(G)}x_v^{t_v}
\]
with nonzero coefficient such that $t_v\le f(v)-1$ for every $v\in V(G)$. When $f(v)=k$ for all $v\in V(G)$, we simply say that $G$ is \emph{Alon--Tarsi $k$-choosable}. The \emph{Alon--Tarsi number} of $G$ is defined by
\[
AT(G)=\min\{k:\text{$G$ is Alon--Tarsi $k$-choosable}\}.
\]

The celebrated Combinatorial Nullstellensatz is stated as follows.

\begin{theorem}[\cite{AT92}]\label{cnull}
\textnormal{(Combinatorial Nullstellensatz)}
Let $\mathbb{F}$ be a field, and let
$f(x_1,\ldots,x_n)\in\mathbb{F}[x_1,\ldots,x_n]$.
Suppose $\deg(f)=\sum_{i=1}^n t_i$, where each $t_i\ge0$, and the coefficient of
$\prod_{i=1}^n x_i^{t_i}$ is nonzero.
If $S_1,\ldots,S_n\subseteq\mathbb{F}$ satisfy $|S_i|\ge t_i+1$ for every $i$,
then there exist $s_i\in S_i$ such that
$f(s_1,\ldots,s_n)\neq0$.
\end{theorem}

Theorem~\ref{cnull} immediately implies that $\chi_{\ell}(G)\le AT(G)$ for every graph $G$. Therefore, any upper bound on $AT(G)$ yields the corresponding upper bound on $\chi_{\ell}(G)$. For example, Thomassen \cite{Thomassen-94} proved that every planar graph is $5$-choosable. Later, Zhu \cite{Zhu-AT} established the stronger result that $AT(G)\le5$ for every planar graph $G$, which immediately implies Thomassen's theorem.

In this paper, we study the Alon--Tarsi number of squares of subcubic planar graphs. The \emph{square} of a graph $G$, denoted by $G^2$, is obtained by joining every pair of vertices whose distance in $G$ is at most $2$, while keeping the vertex set unchanged.

Coloring squares of graphs has received considerable attention following Wegner's conjecture, which proposes upper bounds on $\chi(G^2)$ in terms of the maximum degree of a planar graph $G$.
 For further results related to Wegner's conjecture, we refer the reader to \cite{Cranston22, Hartke, JKK, Thomassen}.

A graph is \emph{subcubic} if its maximum degree is at most $3$, and a cycle of length $k$ is called a \emph{$k$-cycle}. Recently, Kim and Luo \cite{KLuo25} proved the following theorem.

\begin{theorem}\label{thm-KLuo25}
If $G$ is a subcubic planar graph containing no $k$-cycles for $4\le k\le8$, then $\chi_{\ell}(G^2)\le6$.
\end{theorem}

Since $\chi_{\ell}(G^2)\le AT(G^2)$ for every graph $G$, Theorem~\ref{thm-KLuo25} naturally raises the question of whether the stronger conclusion also holds for the Alon--Tarsi number. In this paper, we answer this question affirmatively.

\begin{theorem}\label{main-thm}
If $G$ is a subcubic planar graph containing no $k$-cycles for $4\le k\le8$, then $AT(G^2)\le6$.
\end{theorem}

\medskip

The proof combines the structural framework developed by Kim and Luo \cite{KLuo25} with new reducibility arguments for the Alon--Tarsi number. In Section~\ref{sec:prelim}, we first introduce the structural tools needed for studying the Alon--Tarsi number and the relevant reducible configurations. We then prove Theorem~\ref{main-thm} assuming the reducibility of these configurations. The reducibility proofs are presented in Section~\ref{sec:reducibility}.

 
 \section{Alon--Tarsi tools and the proof of Theorem~\ref{main-thm}}
 \label{sec:prelim}

In this section, we introduce an alternative characterization of the Alon--Tarsi number and establish the reducibility of the configurations arising in the structural analysis of Kim and Luo~\cite{KLuo25}.

\subsection{An alternative characterization of the Alon--Tarsi number}\label{sec:def}

We first present an alternative characterization of the Alon--Tarsi number in terms of graph orientations.

A digraph $D$ is \emph{Eulerian} if $d_D^{+}(v)=d_D^{-}(v)$ for every vertex $v$, where $d_D^{+}(v)$ and $d_D^{-}(v)$ denote the outdegree and indegree of $v$ in $D$, respectively. An Eulerian digraph is \emph{even} if it has an even number of edges, and \emph{odd} otherwise. For a digraph $D$, let $EE(D)$ and $EO(D)$ denote the sets of even and odd spanning Eulerian subdigraphs of $D$, respectively. An orientation $D$ of a graph is called an \emph{Alon--Tarsi orientation} (or simply an \emph{AT-orientation}) if
\[
|EE(D)|\neq |EO(D)|.
\]

Alon and Tarsi~\cite{AT92} established the following connection between graph-polynomial coefficients and Eulerian subdigraphs.

\begin{theorem}[\cite{AT92}]
Let $D$ be an orientation of a graph $G$, and let $t_v=d_D^{+}(v)$ for every vertex $v\in V(G)$. Then the absolute value of the coefficient of the monomial
\[
\prod_{v\in V(G)}x_v^{t_v}
\]
in the expansion of $P_G(\bm{x})$ is equal to
\[
\bigl||EE(D)|-|EO(D)|\bigr|.
\]
\end{theorem}

The above theorem immediately implies that, for a function
$f:V(G)\rightarrow\{0,1,2,\ldots\}$,
the graph polynomial $P_G(\bm{x})$ contains a monomial
\[
\prod_{v\in V(G)}x_v^{t_v}
\]
with nonzero coefficient satisfying $t_v\le f(v)-1$ for every $v\in V(G)$ if and only if $G$ admits an AT-orientation $D$ with
$d_D^{+}(v)\le f(v)-1$ for every vertex $v\in V(G)$. Consequently, the Alon--Tarsi number can be characterized as
\[
AT(G)=\min\left\{k:\text{$G$ admits an AT-orientation $D$ with }
d_D^{+}(v)\le k-1 \text{ for every }v\in V(G)\right\}.
\]

We next state a lemma due to Lu and Zhu~\cite{LZ19}, which will play a key role in proving the reducibility of configurations. For completeness, we include its proof.

\begin{lemma}[\cite{LZ19}] \label{key-AT-configuration}
Assume that $D$ is a digraph with $V(D)=X\cup Y$ and $X\cap Y=\emptyset$. If every arc between $X$ and $Y$ is oriented from $X$ to $Y$, then $D$ is an AT-orientation if and only if both $D[X]$ and $D[Y]$ are AT-orientations.
\end{lemma}

\begin{proof}
Let $D_1=D[X]$ and $D_2=D[Y]$. Since every arc between $X$ and $Y$ is oriented from $X$ to $Y$, none of these arcs belongs to a directed cycle. Consequently, none of them is contained in any Eulerian subdigraph of $D$.

It follows that every Eulerian subdigraph $H$ of $D$ is the arc-disjoint union of an Eulerian subdigraph $H_1$ of $D_1$ and an Eulerian subdigraph $H_2$ of $D_2$. Moreover, $H$ is even if and only if $H_1$ and $H_2$ have the same parity. Therefore,
\begin{eqnarray*}
|EE(D)| &=& |EE(D_1)|\,|EE(D_2)|
          + |EO(D_1)|\,|EO(D_2)|,\\
|EO(D)| &=& |EE(D_1)|\,|EO(D_2)|
          + |EO(D_1)|\,|EE(D_2)|.
\end{eqnarray*}
Hence,
\[
|EE(D)|-|EO(D)|
=
\bigl(|EE(D_1)|-|EO(D_1)|\bigr)
\bigl(|EE(D_2)|-|EO(D_2)|\bigr).
\]

Therefore,
\[
|EE(D)|-|EO(D)|\neq0
\]
if and only if
\[
|EE(D_1)|-|EO(D_1)|\neq0
\quad\text{and}\quad
|EE(D_2)|-|EO(D_2)|\neq0.
\]
The conclusion follows.
\end{proof}

\subsection{Proof of Theorem~\ref{main-thm}}
\label{section-redu}

We prove Theorem~\ref{main-thm} by contradiction. Let $G$ be a minimal counterexample; that is,
$AT(G^2)>6$, whereas $AT(H^2)\le6$ for every proper subgraph $H$ of $G$.
We further assume that $G$ is embedded in the plane. Clearly, $G$ is connected.

The following claim is a direct consequence of Lemma~\ref{key-AT-configuration} and will be used repeatedly in proving the reducibility of configurations.

\begin{claim}\label{key-remark}
Let $G_1$ be an induced subgraph of $G$, and define
\[
f(v)=6-t_v,
\]
where
\[
t_v=\bigl|N_{G^2}(v)\setminus V(G_1)\bigr|
\]
for each $v\in V(G_1)$. Then $G_1^2$ is not Alon--Tarsi $f$-choosable.
\end{claim}

\begin{proof}
Suppose, to the contrary, that $G_1^2$ is Alon--Tarsi $f$-choosable. Then $G_1^2$ admits an AT-orientation $D_1$ satisfying
\[
d_{D_1}^+(v)\le f(v)-1=5-t_v
\]
for every $v\in V(G_1^2)$.

Let $G_2=G-V(G_1)$. By the minimality of $G$, the graph $G_2^2$ admits an AT-orientation $D_2$ such that
\[
d_{D_2}^+(u)\le5
\]
for every $u\in V(G_2^2)$.

Now orient $G^2$ by combining $D_1$ and $D_2$, and orient every edge between $V(G_1)$ and $V(G_2)$ from $V(G_1)$ to $V(G_2)$. Denote the resulting orientation by $D$.

For each $v\in V(G_1^2)$,
\[
d_D^+(v)\le (5-t_v)+t_v=5,
\]
and for each $u\in V(G_2^2)$,
\[
d_D^+(u)=d_{D_2}^+(u)\le5.
\]
By Lemma~\ref{key-AT-configuration}, $D$ is an AT-orientation of $G^2$. Hence
\[
AT(G^2)\le6,
\]
contradicting the choice of $G$.
\end{proof}

\medskip

We next extend the structural results of Kim and Luo~\cite{KLuo25} from list coloring to Alon--Tarsi choosability. More precisely, we show that every configuration proved reducible for list coloring in~\cite{KLuo25} is also reducible for Alon--Tarsi choosability. Consequently, once these reducibility results are established, the remainder of the proof in~\cite{KLuo25} carries over verbatim to prove Theorem~\ref{main-thm}.

We begin with the following simple observations.

\begin{claim}\label{CL:2-vertex}
\begin{enumerate}[(1)]
\item  The minimum degree of $G$  is at least 2.

\item No $2$-vertex of $G$ is contained in a triangle.

\item No two $2$-vertices of $G$ are adjacent.
\end{enumerate}
\end{claim}

\begin{proof}
(1) Suppose that $G$ contains a $1$-vertex $x$, and let $H=\{x\}$. Then
\[
t_x=\bigl|N_{G^2}(x)\setminus V(H)\bigr|\le3,
\]
so $f(x)\ge3$. Since $H^2$ consists of a single vertex, it is Alon--Tarsi $f$-choosable, contradicting Claim~\ref{key-remark}.

(2) Suppose that $G$ contains a $2$-vertex $x$ lying on a triangle $xyzx$, and let $H=\{x\}$. Then
\[
t_x=\bigl|N_{G^2}(x)\setminus V(H)\bigr|\le4,
\]
so $f(x)\ge2$. Again, $H^2$ is Alon--Tarsi $f$-choosable, contradicting Claim~\ref{key-remark}.

(3) Suppose that $G$ contains adjacent $2$-vertices $v_1$ and $v_2$, and let
\[
H=\{v_1,v_2\}.
\]
Then
\[
t_{v_i}=\bigl|N_{G^2}(v_i)\setminus V(H)\bigr|\le4
\]
for $i=1,2$, so $f(v_i)\ge2$ for each $i$. Since $H^2$ is a single edge, it is Alon--Tarsi $f$-choosable, contradicting Claim~\ref{key-remark}.
\end{proof}
\begin{figure*}[htbp]
\begin{multicols}{3}
\begin{center}
\begin{tikzpicture}
[u/.style={fill=black, minimum size =3pt,ellipse,inner sep=1pt},
u-circle/.style={circle, draw, fill=none, minimum size =3pt,ellipse,inner sep=1pt},node distance=1.5cm,scale=0.8]

\node[u-circle] (v1) at (1, 0){};
\node[u] (v2) at (2, 0){};
\node[u] (v3) at (3, 0){};
\node[u-circle] (v4) at (4, 0){};
\node[u] (v5) at (2.5, 1){};

  \draw (0.5, 0) -- (v1);
  \draw (v1) -- (v2);
  \draw (v2) -- (v3);
  \draw (v3) -- (v4);
  \draw (v4) -- (4.5, 0);

  \draw (v2) -- (v5);
  \draw (v3) -- (v5);

   \node[below=0.1cm, font=\small] at (v1) {$v_1$};
   \node[below=0.1cm,font=\small] at (v2) {$v_2$};
   \node[below=0.1cm, font=\small] at (v3) {$v_3$};
   \node[below=0.1cm,font=\small] at (v4) {$v_4$};
    \node[above=0.1cm,font=\small] at (v5) {$v_5$};

 \end{tikzpicture}
        \vfill {\small The subgraph $T_1$}
\end{center}
\par
\begin{center}
\begin{tikzpicture}
[u/.style={fill=black, minimum size =3pt,ellipse,inner sep=1pt},
u-circle/.style={circle, draw, fill=none, minimum size =3pt,ellipse,inner sep=1pt},node distance=1.5cm,scale=0.8]

\node[u] (v1) at (1, 0){};
\node[u] (v2) at (2, 0){};
\node[u] (v3) at (3, 0){};
\node[u] (v4) at (4, 0){};
\node[u-circle] (v5) at (5, 0){};
\node[u] (v6) at (1.5, 1){};
\node[u] (v7) at (3.5, 1){};

  \draw (0.5, 0) -- (v1);
  \draw (v1) -- (v2);
  \draw (v2) -- (v3);
  \draw (v3) -- (v4);
  \draw (v4) -- (v5);
  \draw (v5) -- (5.5, 0);

  \draw (v1) -- (v6);
  \draw (v2) -- (v6);
  \draw (v3) -- (v7);
  \draw (v4) -- (v7);

   \node[below=0.1cm, font=\small] at (v1) {$v_6$};
   \node[below=0.1cm,font=\small] at (v2) {$v_1$};
   \node[below=0.1cm, font=\small] at (v3) {$v_2$};
   \node[below=0.1cm,font=\small] at (v4) {$v_3$};

    \node[below=0.1cm, font=\small] at (v5) {$v_4$};
     \node[above=0.1cm, font=\small] at (v6) {$v_7$};
     \node[above=0.1cm, font=\small] at (v7) {$v_5$};

 \end{tikzpicture}
        \vfill {\small The subgraph  $T_2$}
\end{center}
\par
\begin{center}
\begin{tikzpicture}
[u/.style={fill=black, minimum size =3pt,ellipse,inner sep=1pt},
u-circle/.style={circle, draw, fill=none, minimum size =3pt,ellipse,inner sep=1pt},node distance=1.5cm,scale=0.8]

\node[u] (v1) at (1, 0){};
\node[u] (v2) at (2, 0){};
\node[u-circle] (v3) at (3, 0){};
\node[u] (v4) at (4, 0){};
\node[u-circle] (v5) at (5, 0){};
\node[u] (v6) at (1.5, 1){};

  \draw (v1) -- (v2);
  \draw (v2) -- (v3);
  \draw (v3) -- (v4);
  \draw (v4) -- (v5);
  \draw (v5) -- (5.5, 0);

  \draw (v1) -- (v6);
  \draw (v2) -- (v6);
  \draw (v4) -- (4, 0.5);

   \node[below=0.1cm, font=\small] at (v1) {$v_1$};
   \node[below=0.1cm,font=\small] at (v2) {$v_2$};
   \node[below=0.1cm, font=\small] at (v3) {$v_3$};
   \node[below=0.1cm,font=\small] at (v4) {$v_4$};

    \node[below=0.1cm, font=\small] at (v5) {$v_5$};
     \node[above=0.1cm, font=\small] at (v6) {$v_6$};
 \end{tikzpicture}
        \vfill {\small The subgraph  $T_3$}
\end{center}
\end{multicols} 
\caption{Reducible configurations $T_1$, $T_2$, and $T_3$. Black (resp.\ white) vertices denote $3$-vertices (resp.\ $2$-vertices).}
\label{T123-subgraph}
\end{figure*}
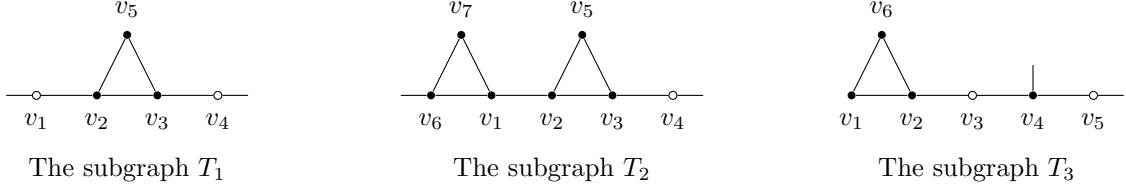

 For a cycle $C$ in $G$, let $d(C)$ denote its length. A \emph{$k^+$-cycle} is a cycle of length at least $k$. Let
$t(C)$ denote the sum of the number of $2$-vertices on $C$ and the number of triangles sharing an edge with $C$.
A key step in our proof is to bound $t(C)$ for cycles of length at least $9$. We begin by ruling out the following configurations.

\begin{claim}
\label{cl:reducible}
For each $i\in\{1,2,3\}$, the graph $G$ contains no copy of the configuration $T_i$ shown in Figure~\ref{T123-subgraph}.
\end{claim}

\begin{proof}
Recall from Claim~\ref{CL:2-vertex} that every vertex contained in a triangle has degree $3$.

\medskip
\noindent
(a)  We first show that $G$ doesn't contain $T_1$ or $T_2$.  Suppose by contradiction that $G$ contains $T_j$ for some $j =1,2$.  Let $Q_j$ be the subgraph of $T_j$ induced by  $\{v_1,v_2,v_3,v_4,v_5\}$ ($Q_1=T_1$).
Again define 
$f(v)=6-t_v$ for each $v\in V(Q_j)$.
Then, as illustrated in  Figure~\ref{T123-color}(i)  or (ii) depending on $j=1$ or $2$, we have
\[
f(v_i)\ge 
\begin{cases}
2,& i=1,\\
3,& i\in\{2, 4,5\},\\
4,& i=3.
\end{cases}
\]

Moreover,
\[
\begin{aligned}
P_{Q_j^2}(\bm{x})
&=\prod_{u\sim v,\;u<v}(x_u-x_v)\\
&=(x_1-x_2)(x_1-x_3)(x_1-x_5)
(x_2-x_3)(x_2-x_4)(x_2-x_5)
(x_3-x_4)(x_3-x_5)(x_4-x_5).
\end{aligned}
\]
A direct computation in \textsc{Mathematica} shows that the coefficient of
\[
x_1x_2x_3^3x_4^2x_5^2
\]
is $-1$. Hence $Q_j^2$ is Alon--Tarsi $f$-choosable, contradicting Claim~\ref{key-remark}. Thus $G$ doesn't contain $T_j$ for each $j=1, 2$.

\medskip
\noindent
(b)  Next we show that $G$ doesn't contain $T_3$. Suppose by contradiction that  $G$ contains $T_3$. Let $Q_3$ be the subgraph of $T_3$ induced by
$\{v_2,v_3,v_4,v_5\}$.
Define $f(v)=6-t_v$ for each $v\in V(Q_3)$.
Then, as illustrated in  Figure~\ref{T123-color}(iii), we have
\[
f(v_i)=
\begin{cases}
2,& i\in\{2,4,5\},\\
3,& i=3.
\end{cases}
\]

The graph polynomial of $Q_3^2$ is
\[
\begin{aligned}
P_{Q_3^2}(\bm{x})
&=\prod_{u\sim v,\;u<v}(x_u-x_v)\\
&=(x_2-x_3)(x_2-x_4)
(x_3-x_4)(x_3-x_5)(x_4-x_5).
\end{aligned}
\]
A direct computation in \textsc{Mathematica} shows that the coefficient of
\[
x_2x_3^2x_4x_5
\]
is $1$. Hence $Q_3^2$ is Alon--Tarsi $f$-choosable, again contradicting Claim~\ref{key-remark}. Therefore $T_3$ cannot occur in $G$.
\end{proof}

\begin{figure*}[htbp]
\begin{multicols}{3}
\begin{center}
\begin{tikzpicture}
[u/.style={fill=black, minimum size =3pt,ellipse,inner sep=1pt},
u-circle/.style={circle, draw, fill=none, minimum size =3pt,ellipse,inner sep=1pt},node distance=1.5cm,scale=0.8]

\node[u-circle] (v1) at (1, 0){};
\node[u] (v2) at (2, 0){};
\node[u] (v3) at (3, 0){};
\node[u-circle] (v4) at (4, 0){};
\node[u] (v5) at (2.5, 1){};

  \draw (0.5, 0) -- (v1);
  \draw (v1) -- (v2);
  \draw (v2) -- (v3);
  \draw (v3) -- (v4);
  \draw (v4) -- (4.5, 0);
  \draw (v2) -- (v5);
  \draw (v3) -- (v5);

   \node[below=0.1cm, font=\small] at (v1) {$v_1$};
   \node[below=0.1cm,font=\small] at (v2) {$v_2$};
   \node[below=0.1cm, font=\small] at (v3) {$v_3$};
   \node[below=0.1cm,font=\small] at (v4) {$v_4$};
    \node[above=0.1cm,font=\small] at (v5) {$v_5$};

 \node[above=0.05cm,  font=\scriptsize] at (v1) {\textcolor{blue}{3}};
\node[left=0.05cm, above=0.05cm,font=\scriptsize] at (v2) {\textcolor{blue}{4}};
\node[right=0.05cm, above=0.05cm,font=\scriptsize] at (v3) {\textcolor{blue}{4}};
\node[above=0.05cm, font=\scriptsize] at (v4) {\textcolor{blue}{3}};
\node[right=0.05cm, font=\scriptsize] at (v5) {\textcolor{blue}{3}};
 \end{tikzpicture}
        \vfill {(i) $f(v)$ of $T_1$}
\end{center}
\par
\begin{center}
\begin{tikzpicture}
[u/.style={fill=black, minimum size =3pt,ellipse,inner sep=1pt},
u-circle/.style={circle, draw, fill=none, minimum size =3pt,ellipse,inner sep=1pt},node distance=1.5cm,scale=0.8]
\node[u] (v1) at (1, 0){};
\node[u] (v2) at (2, 0){};
\node[u] (v3) at (3, 0){};
\node[u] (v4) at (4, 0){};
\node[u-circle] (v5) at (5, 0){};
\node[u] (v6) at (1.5, 1){};
\node[u] (v7) at (3.5, 1){};

  \draw (0.5, 0) -- (v1);
  \draw (v1) -- (v2);
  \draw (v2) -- (v3);
  \draw (v3) -- (v4);
  \draw (v4) -- (v5);
  \draw (v5) -- (5.5, 0);

  \draw (v1) -- (v6);
  \draw (v2) -- (v6);
  \draw (v3) -- (v7);
  \draw (v4) -- (v7);

 \node[below=0.1cm, font=\small] at (v1) {$v_6$};
   \node[below=0.1cm,font=\small] at (v2) {$v_1$};
   \node[below=0.1cm, font=\small] at (v3) {$v_2$};
   \node[below=0.1cm,font=\small] at (v4) {$v_3$};

    \node[below=0.1cm, font=\small] at (v5) {$v_4$};
     \node[above=0.1cm, font=\small] at (v6) {$v_7$};
     \node[above=0.1cm, font=\small] at (v7) {$v_5$};

 \node[left=0.05cm, above=0.05cm,  font=\scriptsize] at (v1) {};
\node[right=0.05cm, above=0.05cm,font=\scriptsize] at (v2) {\textcolor{blue}{2}};
\node[left=0.05cm, above=0.05cm,font=\scriptsize] at (v3) {\textcolor{blue}{3}};
\node[right=0.05cm, above=0.05cm, font=\scriptsize] at (v4) {\textcolor{blue}{4}};
\node[above=0.05cm, font=\scriptsize] at (v5) {\textcolor{blue}{3}};
\node[right=0.1cm, font=\scriptsize] at (v6) {};
\node[right=0.1cm, font=\scriptsize] at (v7) {\textcolor{blue}{3}};
 \end{tikzpicture}
        \vfill {(ii) $f(v)$  of $Q_2$ ($T_2$ case)}
\end{center}
\par
\begin{center}
\begin{tikzpicture}
[u/.style={fill=black, minimum size =3pt,ellipse,inner sep=1pt},
u-circle/.style={circle, draw, fill=none, minimum size =3pt,ellipse,inner sep=1pt},node distance=1.5cm,scale=0.8]
\node[u] (v1) at (1, 0){};
\node[u] (v2) at (2, 0){};
\node[u-circle] (v3) at (3, 0){};
\node[u] (v4) at (4, 0){};
\node[u-circle] (v5) at (5, 0){};
\node[u] (v6) at (1.5, 1){};

  \draw (v1) -- (v2);
  \draw (v2) -- (v3);
  \draw (v3) -- (v4);
  \draw (v4) -- (v5);
  \draw (v5) -- (5.5, 0);

  \draw (v1) -- (v6);
  \draw (v2) -- (v6);
  \draw (v4) -- (4, 0.5);

   \node[below=0.1cm, font=\small] at (v1) {$v_1$};
   \node[below=0.1cm,font=\small] at (v2) {$v_2$};
   \node[below=0.1cm, font=\small] at (v3) {$v_3$};
   \node[below=0.1cm,font=\small] at (v4) {$v_4$};

    \node[below=0.1cm, font=\small] at (v5) {$v_5$};
     \node[above=0.1cm, font=\small] at (v6) {$v_6$};

 \node[left=0.05cm, above=0.05cm,  font=\scriptsize] at (v1) {};
\node[right=0.05cm, above=0.05cm,font=\scriptsize] at (v2) {\textcolor{blue}{2}};
\node[above=0.05cm,font=\scriptsize] at (v3) {\textcolor{blue}{3}};
\node[right=0.1cm, above=0.05cm, font=\scriptsize] at (v4) {\textcolor{blue}{2}};
\node[above=0.05cm, font=\scriptsize] at (v5) {\textcolor{blue}{2}};
\node[right=0.1cm, font=\scriptsize] at (v6) {};
 \end{tikzpicture}
        \vfill {(iii) $f(v)$  of $Q_3$ ($T_3$ case)}
\end{center}
\end{multicols} 
\caption{The numbers indicate the values of $f(v)$ at the corresponding vertices.}
\label{T123-color}
\end{figure*}
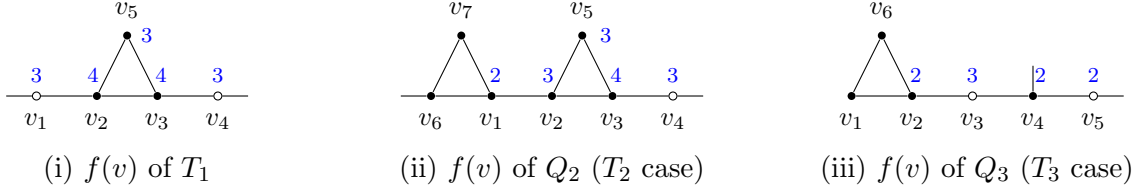

  As shown in \cite{KLuo25} (see Definition~12 and Claims~13--15), Claim~\ref{cl:reducible} implies the following result.

\begin{claim}
\label{lemma-C11}
If $C$ is a $9^+$-cycle in $G$, then
\[
t(C)\le d(C)-\left\lceil\frac{d(C)}{2}\right\rceil.
\]
In particular,
\[
t(C)\le d(C)-6
\]
whenever $d(C)\ge11$.
\end{claim}

The proof is identical to that of \cite[Definition~12 and Claims~13--15]{KLuo25}, and is therefore omitted.

To complete the discharging argument, it remains to extend the bound
\[
t(C)\le d(C)-6
\]
from $11^+$-cycles in Claim~\ref{lemma-C11} to $9$- and $10$-cycles. For this purpose, we establish the following reducible configurations.

\begin{claim}[Reducible configurations on $10$-cycles]
\label{lem-H234}
The graph $G$ contains none of the configurations $H_1,H_2,H_3,$ or $H_4$ shown in Figure~\ref{H234}.
\end{claim}

\begin{claim}[Reducible configurations on $9$-cycles]
\label{claim-F-one}
The graph $G$ contains none of the configurations $F_1,F_2,\ldots,F_{12}$ shown in Figure~\ref{W12-subgraph}.
\end{claim}

The proofs of Claims~\ref{lem-H234} and~\ref{claim-F-one} are deferred to Section~\ref{sec:reducibility}.

Using Claims~18 and 20 of \cite{KLuo25} together with Claims~\ref{cl:reducible}, \ref{lem-H234}, and~\ref{claim-F-one}, one obtains the following result by exactly the same argument as in \cite{KLuo25}.

\begin{claim}\label{claim-main}
For every face $f$ of $G$ with $d(f)\ge9$,
\[
t(f)\le d(f)-6.
\]
\end{claim}

Again, we omit the proof since it is identical to that in \cite{KLuo25}.

We are now ready to complete the proof by the discharging method. Assign each vertex $v$ the initial charge
\[
\omega(v)=2d(v)-6,
\]
and each face $f$ the initial charge
\[
\omega(f)=d(f)-6.
\]
By Euler's formula,
\[
|V(G)|-|E(G)|+|F(G)|=2,
\]
we obtain
\begin{equation*}
\sum_{x\in V(G)\cup F(G)}\omega(x)
=
\sum_{v\in V(G)}(2d(v)-6)
+
\sum_{f\in F(G)}(d(f)-6)
=
-12.
\end{equation*}

We redistribute the initial charges according to the following rule.

\medskip
\noindent\textbf{Discharging Rule}

\medskip
\noindent
(R)
Each $9^+$-face sends one unit of charge to each incident $2$-vertex and to each adjacent $3$-face.

\medskip

Let $\omega'(x)$ denote the final charge of each vertex or face $x$. We show that $\omega'(x)\ge0$ for every $x\in V(G)\cup F(G)$.

If $x$ is a $2$-vertex, then it receives one unit of charge from each of its two incident $9^+$-faces by (R). Hence
\[
\omega'(x)=2\cdot2-6+2=0.
\]
If $x$ is a $3$-vertex, then it neither sends nor receives charge, and thus
\[
\omega'(x)=\omega(x)=0.
\]

Now let $f$ be a face of $G$.
If $d(f)\ge9$, then Claim~\ref{claim-main} implies that $t(f)\le d(f)-6$. Therefore,
\[
\omega'(f)
=
d(f)-6-t(f)
\ge0.
\]

If $d(f)=3$, then $f$ receives one unit of charge from each of its three adjacent $9^+$-faces, and hence
\[
\omega'(f)
=
3-6+3
=
0.
\]

Thus every vertex and every face has nonnegative final charge. Consequently,
\[
0
\le
\sum_{x\in V(G)\cup F(G)}\omega'(x)
=
\sum_{x\in V(G)\cup F(G)}\omega(x)
=
-12,
\]
a contradiction.

This completes the proof of Theorem~\ref{main-thm}.
\section{Proofs of Claims~\ref{lem-H234} and~\ref{claim-F-one}}
\label{sec:reducibility}

In this section, we prove Claims~\ref{lem-H234} and~\ref{claim-F-one}.

\subsection{Proof of Claim~\ref{lem-H234}: Reducible configurations on $10$-cycles}

We first prove Claim~\ref{lem-H234} by showing that none of the configurations $H_1$, $H_2$, $H_3$, and $H_4$ can occur in $G$.

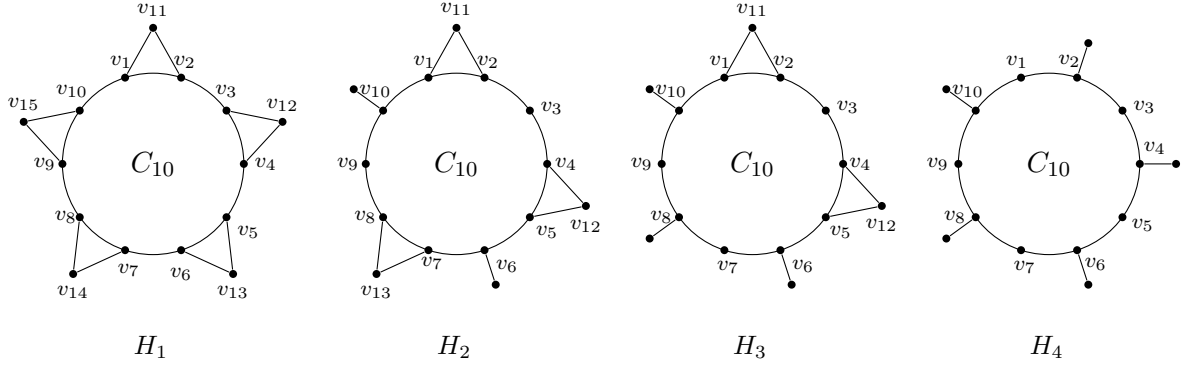
\begin{figure}[htbp]
\begin{center}
\subfigure{\begin{tikzpicture}
[u/.style={fill=black, minimum size =3pt,ellipse,inner sep=1pt},node distance=1.5cm,scale=1.2]
\node[u] (v1) at (108:1){};
\node[u] (v2) at (72:1){};
\node[u] (v3) at (36:1){};
\node[u] (v4) at (0:1){};
\node[u] (v5) at (324:1){};
\node[u] (v6) at (288:1){};
\node[u] (v7) at (252:1){};
\node[u] (v8) at (216:1){};
\node[u] (v9) at (180:1){};
\node[u] (v10) at (144:1){};
\node[u] (v11) at (90:1.5){};
\node[u] (v12) at (18:1.5){};
\node[u] (v13) at (306:1.5){};
\node[u] (v14) at (234:1.5){};
\node[u] (v15) at (162:1.5){};
\node (C10) at (0, 0){$C_{10}$};

\draw   (0,0) circle[radius=1cm];
  \draw (v1) -- (v11);
  \draw (v2) -- (v11);
  \draw (v3) -- (v12);
  \draw (v4) -- (v12);
  \draw (v5) -- (v13);
  \draw (v6) -- (v13);
  \draw (v7) -- (v14);
  \draw (v8) -- (v14);
  \draw (v9) -- (v15);
  \draw (v10) -- (v15);

   \node[left=0.05cm, above=-0.01cm, font=\scriptsize] at (v1) {$v_1$};
   \node[ right=0.05cm, above=-0.01cm,  font=\scriptsize] at (v2) {$v_2$};
   \node[above=0.05cm, font=\scriptsize] at (v3) {$v_3$};
   \node[right=0.01cm, font=\scriptsize] at (v4) {$v_4$};
    \node[below=0.2cm,right=0.01cm, font=\scriptsize] at (v5) {$v_5$};
   \node[below=0.1cm,font=\scriptsize] at (v6) {$v_6$};
   \node[right=0.05cm, below=0.05cm,  font=\scriptsize] at (v7) {$v_7$};
   \node[left=-0.1cm, font=\scriptsize] at (v8) {$v_8$};
    \node[left=-0.05cm,font=\scriptsize] at (v9) {$v_9$};
   \node[left=0.1cm, above=0.05cm,  font=\scriptsize] at (v10) {$v_{10}$};
    \node[above=0.02cm, font=\scriptsize] at (v11) {$v_{11}$};
    \node[above=0.01cm,  font=\scriptsize] at (v12) {$v_{12}$};
    \node[below=0.02cm, font=\scriptsize] at (v13) {$v_{13}$};
     \node[below=0.02cm,  font=\scriptsize] at (v14) {$v_{14}$};
    \node[above=0.02cm, font=\scriptsize] at (v15) {$v_{15}$};

 \node[below=0.05cm,  font=\scriptsize] at (v1) {};
\node[below=0.05cm,  font=\scriptsize] at (v2) {};
\node[left=0.05cm,font=\scriptsize] at (v3) {};
\node[above=0.05cm, left=0.05cm,font=\scriptsize] at (v4) {};
\node[left=0.05cm,font=\scriptsize] at (v5) {};
\node[above=0.05cm, font=\scriptsize] at (v6) {};
\node[above=0.05cm, font=\scriptsize] at (v7) {};
\node[right=0.05cm, font=\scriptsize] at (v8) {};
\node[right=0.05cm, font=\scriptsize] at (v9) {};
\node[right=0.05cm,  font=\scriptsize] at (v10) {};
\node[right=0.1cm, font=\scriptsize] at (v11) {};
\node[above=0.1cm, font=\scriptsize] at (v12) {};
\node[right=0.1cm, font=\scriptsize] at (v13) {};
\node[left=0.1cm, font=\scriptsize] at (v14) {};
\node[below=0.1cm, font=\scriptsize] at (v15) {};
\node[left=0.4cm,below=1cm] at (v6) {\small $H_1$};
 \end{tikzpicture}}
\subfigure{
\begin{tikzpicture}
[u/.style={fill=black, minimum size =3pt,ellipse,inner sep=1pt},node distance=1.5cm,scale=1.2]
\node[u] (v1) at (108:1){};
\node[u] (v2) at (72:1){};
\node[u] (v3) at (36:1){};
\node[u] (v4) at (0:1){};
\node[u] (v5) at (324:1){};
\node[u] (v6) at (288:1){};
\node[u] (v7) at (252:1){};
\node[u] (v8) at (216:1){};
\node[u] (v9) at (180:1){};
\node[u] (v10) at (144:1){};
\node[u] (v11) at (90:1.5){};
\node[u] (v12) at (342:1.5){};
\node[u] (v13) at (234:1.5){};
\node (C10) at (0, 0){$C_{10}$};

\draw   (0,0) circle[radius=1cm];
  \draw (v1) -- (v11);
  \draw (v2) -- (v11);
  \draw (v4) -- (v12);
  \draw (v5) -- (v12);
  \draw (v7) -- (v13);
  \draw (v8) -- (v13);

\node[u] (w6) at (288:1.4){};
\node[u] (w10) at (144:1.4){};
  \draw (v6) -- (w6);
  \draw (v10) -- (w10);

   \node[left=0.05cm, above=-0.01cm, font=\scriptsize] at (v1) {$v_1$};
   \node[ right=0.05cm, above=-0.01cm,  font=\scriptsize] at (v2) {$v_2$};
   \node[above=0.03cm,  right=0.01cm,font=\scriptsize] at (v3) {$v_3$};
   \node[right=-0.03cm,font=\scriptsize] at (v4) {$v_4$};
    \node[below=0.2cm,right=-0.05cm, font=\scriptsize] at (v5) {$v_5$};
   \node[right=0.3cm, below=-0.01cm, font=\scriptsize] at (v6) {$v_6$};
   \node[right=0.05cm, below=-0.05cm,  font=\scriptsize] at (v7) {$v_7$};
   \node[left=-0.05cm, font=\scriptsize] at (v8) {$v_8$};
    \node[left=-0.05cm,font=\scriptsize] at (v9) {$v_9$};
   \node[left=0.1cm, above=0.05cm,  font=\scriptsize] at (v10) {$v_{10}$};
    \node[above=0.05cm,  above=-0.01cm, font=\scriptsize] at (v11) {$v_{11}$};
    \node[below=0.05cm,  font=\scriptsize] at (v12) {$v_{12}$};
    \node[below=0.02cm, font=\scriptsize] at (v13) {$v_{13}$};

 \node[below=0.05cm,  font=\scriptsize] at (v1) {};
\node[below=0.05cm,  font=\scriptsize] at (v2) {};
\node[left=0.05cm,font=\scriptsize] at (v3) {};
\node[above=0.05cm, left=0.05cm,font=\scriptsize] at (v4) {};
\node[left=0.05cm,font=\scriptsize] at (v5) {};
\node[above=0.05cm, font=\scriptsize] at (v6) {};
\node[above=0.05cm, font=\scriptsize] at (v7) {};
\node[right=0.05cm, font=\scriptsize] at (v8) {};
\node[right=0.05cm, font=\scriptsize] at (v9) {};
\node[right=0.05cm,  font=\scriptsize] at (v10) {};
\node[right=0.1cm, font=\scriptsize] at (v11) {};
\node[above=0.1cm, font=\scriptsize] at (v12) {};
 \node[left=0.4cm,below=1cm]  at (v6) {\small $H_2$};
 \end{tikzpicture}}
\subfigure{\begin{tikzpicture}
[u/.style={fill=black, minimum size =3pt,ellipse,inner sep=1pt},node distance=1.5cm,scale=1.2]
\node[u] (v1) at (108:1){};
\node[u] (v2) at (72:1){};
\node[u] (v3) at (36:1){};
\node[u] (v4) at (0:1){};
\node[u] (v5) at (324:1){};
\node[u] (v6) at (288:1){};
\node[u] (v7) at (252:1){};
\node[u] (v8) at (216:1){};
\node[u] (v9) at (180:1){};
\node[u] (v10) at (144:1){};
\node[u] (v11) at (90:1.5){};
\node[u] (v12) at (342:1.5){};

\node (C10) at (0, 0){$C_{10}$};

\draw   (0,0) circle[radius=1cm];
  \draw (v1) -- (v11);
  \draw (v2) -- (v11);
  \draw (v4) -- (v12);
  \draw (v5) -- (v12);

\node[u] (w6) at (288:1.4){};
\node[u] (w8) at (216:1.4){};
\node[u] (w10) at (144:1.4){};

  \draw (v6) -- (w6);
  \draw (v8) -- (w8);
  \draw (v10) -- (w10);

       \node[left=0.05cm, above=-0.01cm, font=\scriptsize] at (v1) {$v_1$};
   \node[ right=0.05cm, above=-0.01cm,  font=\scriptsize] at (v2) {$v_2$};
   \node[above=0.05cm, right=0.01cm, font=\scriptsize] at (v3) {$v_3$};
   \node[right=-0.05cm,font=\scriptsize] at (v4) {$v_4$};
    \node[below=0.2cm,right=-0.08cm, font=\scriptsize] at (v5) {$v_5$};
   \node[right=0.3cm, below=-0.01cm, font=\scriptsize] at (v6) {$v_6$};
   \node[right=0.05cm, below=0.01cm,  font=\scriptsize] at (v7) {$v_7$};
   \node[left=-0.05cm, font=\scriptsize] at (v8) {$v_8$};
    \node[left=-0.02cm,font=\scriptsize] at (v9) {$v_9$};
   \node[left=0.1cm, above=0.05cm,  font=\scriptsize] at (v10) {$v_{10}$};
    \node[above=0.02cm, font=\scriptsize] at (v11) {$v_{11}$};
    \node[below=0.03cm,font=\scriptsize] at (v12) {$v_{12}$};

 \node[below=0.05cm,  font=\scriptsize] at (v1) {};
\node[below=0.05cm,  font=\scriptsize] at (v2) {};
\node[left=0.05cm,font=\scriptsize] at (v3) {};
\node[above=0.05cm, left=0.05cm,font=\scriptsize] at (v4) {};
\node[left=0.05cm,font=\scriptsize] at (v5) {};
\node[above=0.05cm, font=\scriptsize] at (v6) {};
\node[above=0.05cm, font=\scriptsize] at (v7) {};
\node[right=0.05cm, font=\scriptsize] at (v8) {};
\node[right=0.05cm, font=\scriptsize] at (v9) {};
\node[right=0.05cm,  font=\scriptsize] at (v10) {};
\node[right=0.1cm, font=\scriptsize] at (v11) {};
\node[above=0.1cm, font=\scriptsize] at (v12) {};

 \node[left=0.4cm,below=1cm] at (v6) { \small  $H_3$};
 \end{tikzpicture}}
\subfigure{\begin{tikzpicture}
[u/.style={fill=black, minimum size =3pt,ellipse,inner sep=1pt},node distance=1.5cm,scale=1.2]
\node[u] (v1) at (108:1){};
\node[u] (v2) at (72:1){};
\node[u] (v3) at (36:1){};
\node[u] (v4) at (0:1){};
\node[u] (v5) at (324:1){};
\node[u] (v6) at (288:1){};
\node[u] (v7) at (252:1){};
\node[u] (v8) at (216:1){};
\node[u] (v9) at (180:1){};
\node[u] (v10) at (144:1){};
\node (C10) at (0, 0){$C_{10}$};

\draw   (0,0) circle[radius=1cm];

 \node[u] (w2) at (72:1.4){};
\node[u] (w4) at (0:1.4){};
\node[u] (w6) at (288:1.4){};
\node[u] (w8) at (216:1.4){};
\node[u] (w10) at (144:1.4){};
  \draw (v2) -- (w2);
  \draw (v4) -- (w4);
  \draw (v6) -- (w6);
  \draw (v8) -- (w8);
  \draw (v10) -- (w10);

    \node[left=0.05cm, above=-0.01cm, font=\scriptsize] at (v1) {$v_1$};
   \node[left=0.1cm, above=0.01cm,  font=\scriptsize] at (v2) {$v_2$};
   \node[above=0.05cm, right=0.01cm, font=\scriptsize] at (v3) {$v_3$};
   \node[right=0.2cm, above=-0.01cm, font=\scriptsize] at (v4) {$v_4$};
    \node[below=0.1cm,right=-0.01cm, font=\scriptsize] at (v5) {$v_5$};
   \node[right=0.25cm, below=-0.03cm, font=\scriptsize] at (v6) {$v_6$};
   \node[right=0.05cm, below=0.01cm,  font=\scriptsize] at (v7) {$v_7$};
   \node[left=-0.01cm, font=\scriptsize] at (v8) {$v_8$};
    \node[left=-0.01cm,font=\scriptsize] at (v9) {$v_9$};
   \node[left=0.1cm, above=0.05cm,  font=\scriptsize] at (v10) {$v_{10}$};

 \node[below=0.05cm,  font=\scriptsize] at (v1) {};
\node[below=0.05cm,  font=\scriptsize] at (v2) {};
\node[left=0.05cm,font=\scriptsize] at (v3) {};
\node[above=0.05cm, left=0.05cm,font=\scriptsize] at (v4) {};
\node[left=0.05cm,font=\scriptsize] at (v5) {};
\node[above=0.05cm, font=\scriptsize] at (v6) {};
\node[above=0.05cm, font=\scriptsize] at (v7) {};
\node[right=0.05cm, font=\scriptsize] at (v8) {};
\node[right=0.05cm, font=\scriptsize] at (v9) {};
\node[right=0.05cm,  font=\scriptsize] at (v10) {};
\node[right=0.1cm, font=\scriptsize] at (v11) {};
\node[above=0.1cm, font=\scriptsize] at (v12) {};

\node[left=0.4cm,below=1cm] at (v6) {\small $H_4$};
 \end{tikzpicture} }
\end{center}
\caption{The four reducible configurations on $10$-cycles.}
\label{H234}
\end{figure}


\noindent\textbf{\textit{Proof of Claim~\ref{lem-H234}.}}
Suppose, to the contrary, that $G$ contains one of the configurations $H_i$, where $i\in\{1,2,3,4\}$.

For each $i\in\{1,2,3,4\}$, let $G_1=G-V(H_i)$. Since no vertex of $H_i$ has more than one neighbor outside $H_i$, it follows that, for any two vertices $u,v\in V(G_1)$,
\[
d_G(u,v)\le2 \quad\Longrightarrow\quad d_{G_1}(u,v)\le2.
\]
Moreover, because $G$ contains no $4$- to $8$-cycles, any two $3$-vertices of $H_i$ that do not lie on the $10$-cycle are at distance at least $3$ in $G$. 

\begin{enumerate}[(a)]
\item \textbf{Case $H_i=H_1$.}

Assume that $G$ contains $H_1$ (see Figure~\ref{H234}). Let $W_1$ be the subgraph of $H_1$ induced by
\[
V(H_1)\setminus\{v_{12},v_{13},v_{14},v_{15}\}.
\]
Thus,
\[
V(W_1)=\{v_1,v_2,\ldots,v_{11}\}.
\]

For each vertex $v\in V(W_1)$, define
\[
f(v)=6-t_v,
\]
where
\[
t_v=\bigl|N_{G^2}(v)\setminus V(W_1)\bigr|.
\]
Then (see Figure~\ref{C10-H1-H2}(a))
\[
f(v_i)=
\begin{cases}
3,& i\in\{4,5,6,7,8,9,11\},\\[2pt]
4,& i\in\{1,2,3,10\}.
\end{cases}
\]

The graph polynomial of $W_1^2$ is
\begin{eqnarray*}
P_{W_1^2}(\bm{x})
&=&\prod_{u\sim v,\;u<v}(x_u-x_v)\\
&=&(x_1-x_2)(x_1-x_3)(x_1-x_9)(x_1-x_{10})(x_1-x_{11})
(x_2-x_3)(x_2-x_4)(x_2-x_{10})\\
&&
(x_2-x_{11})(x_3-x_4)(x_3-x_5)(x_3-x_{11})
(x_4-x_5)(x_4-x_6)(x_5-x_6)(x_5-x_7)\\
&&
(x_6-x_7)(x_6-x_8)(x_7-x_8)(x_7-x_9)
(x_8-x_9)(x_8-x_{10})(x_9-x_{10})(x_{10}-x_{11}).
\end{eqnarray*}

A direct computation in \textsc{Mathematica} shows that the coefficient of
\[
x_1^2x_2^3x_3^3x_4x_5^2x_6^2x_7^2x_8^2x_9^2x_{10}^3x_{11}^2
\]
is $1$. Hence $W_1^2$ is Alon--Tarsi $f$-choosable, a contradiction to Claim~\ref{key-remark}. Therefore $W_1$ cannot occur in $G$, and consequently neither can $H_1$.

\begin{figure*}[htbp]
\begin{multicols}{2}
\begin{center}
\begin{tikzpicture}
[u/.style={fill=black, minimum size =3pt,ellipse,inner sep=1pt},node distance=1.5cm,scale=1.5]
\node[u] (v1) at (108:1){};
\node[u] (v2) at (72:1){};
\node[u] (v3) at (36:1){};
\node[u] (v4) at (0:1){};
\node[u] (v5) at (324:1){};
\node[u] (v6) at (288:1){};
\node[u] (v7) at (252:1){};
\node[u] (v8) at (216:1){};
\node[u] (v9) at (180:1){};
\node[u] (v10) at (144:1){};
\node[u] (v11) at (90:1.5){};
\node[u] (v12) at (18:1.5){};
\node[u] (v13) at (306:1.5){};
\node[u] (v14) at (234:1.5){};
\node[u] (v15) at (162:1.5){};
\node (C10) at (0, 0){$C_{10}$};

\draw   (0,0) circle[radius=1cm];
  \draw (v1) -- (v11);
  \draw (v2) -- (v11);
  \draw (v3) -- (v12);
  \draw (v4) -- (v12);
  \draw (v5) -- (v13);
  \draw (v6) -- (v13);
  \draw (v7) -- (v14);
  \draw (v8) -- (v14);
  \draw (v9) -- (v15);
  \draw (v10) -- (v15);

   \node[left=0.05cm, above=0.05cm, font=\small] at (v1) {$v_1$};
   \node[ right=0.05cm, above=0.05cm,  font=\small] at (v2) {$v_2$};
   \node[above=0.05cm, font=\small] at (v3) {$v_3$};
   \node[right=0.05cm,font=\small] at (v4) {$v_4$};
    \node[below=0.2cm,right=0.01cm, font=\small] at (v5) {$v_5$};
   \node[below=0.1cm,font=\small] at (v6) {$v_6$};
   \node[right=0.05cm, below=0.05cm,  font=\small] at (v7) {$v_7$};
   \node[left=0.1cm, font=\small] at (v8) {$v_8$};
    \node[left=0.05cm,font=\small] at (v9) {$v_9$};
   \node[left=0.1cm, above=0.05cm,  font=\small] at (v10) {$v_{10}$};
    \node[above=0.05cm, font=\small] at (v11) {$v_{11}$};
    \node[above=0.05cm,  font=\small] at (v12) {$v_{12}$};
    \node[below=0.05cm, font=\small] at (v13) {$v_{13}$};
     \node[below=0.05cm,  font=\small] at (v14) {$v_{14}$};
    \node[above=0.05cm, font=\small] at (v15) {$v_{15}$};

 \node[below=0.05cm,  font=\scriptsize] at (v1) {\textcolor{blue}{4}};
\node[below=0.05cm,  font=\scriptsize] at (v2) {\textcolor{blue}{4}};
\node[left=0.05cm,font=\scriptsize] at (v3) {\textcolor{blue}{4}};
\node[above=0.05cm, left=0.05cm,font=\scriptsize] at (v4) {\textcolor{blue}{3}};
\node[left=0.05cm,font=\scriptsize] at (v5) {\textcolor{blue}{3}};
\node[above=0.05cm, font=\scriptsize] at (v6) {\textcolor{blue}{3}};
\node[above=0.05cm, font=\scriptsize] at (v7) {\textcolor{blue}{3}};
\node[right=0.05cm, font=\scriptsize] at (v8) {\textcolor{blue}{3}};
\node[right=0.05cm, font=\scriptsize] at (v9) {\textcolor{blue}{3}};
\node[right=0.05cm,  font=\scriptsize] at (v10) {\textcolor{blue}{4}};
\node[right=0.1cm, font=\scriptsize] at (v11) {\textcolor{blue}{3}};
\node[right=0.1cm, font=\scriptsize] at (v12) {};
\node[right=0.1cm, font=\scriptsize] at (v13) {};
\node[left=0.1cm, font=\scriptsize] at (v14) {};
\node[below=0.1cm, font=\scriptsize] at (v15) {};
\node[left=0.4cm,below=1cm] at (v6) {\small (a) $f(v)$ of $W_1^2$ (subgraph $H_1$)};
 \end{tikzpicture}
\end{center}
\par
\begin{center}
\begin{tikzpicture}
[u/.style={fill=black, minimum size =3pt,ellipse,inner sep=1pt},node distance=1.5cm,scale=1.5]
\node[u] (v1) at (108:1){};
\node[u] (v2) at (72:1){};
\node[u] (v3) at (36:1){};
\node[u] (v4) at (0:1){};
\node[u] (v5) at (324:1){};
\node[u] (v6) at (288:1){};
\node[u] (v7) at (252:1){};
\node[u] (v8) at (216:1){};
\node[u] (v9) at (180:1){};
\node[u] (v10) at (144:1){};
\node[u] (v11) at (90:1.5){};
\node[u] (v12) at (342:1.5){};
\node[u] (v13) at (234:1.5){};
\node (C10) at (0, 0){$C_{10}$};

\draw   (0,0) circle[radius=1cm];
  \draw (v1) -- (v11);
  \draw (v2) -- (v11);
  \draw (v4) -- (v12);
  \draw (v5) -- (v12);
  \draw (v7) -- (v13);
  \draw (v8) -- (v13);

\node[u] (w6) at (288:1.4){};
\node[u] (w10) at (144:1.4){};
  \draw (v6) -- (w6);
  \draw (v10) -- (w10);

   \node[left=0.05cm, above=0.05cm, font=\small] at (v1) {$v_1$};
   \node[ right=0.05cm, above=0.05cm,  font=\small] at (v2) {$v_2$};
   \node[above=0.05cm, font=\small] at (v3) {$v_3$};
   \node[right=0.05cm,font=\small] at (v4) {$v_4$};
    \node[below=0.2cm,right=0.01cm, font=\small] at (v5) {$v_5$};
   \node[right=0.3cm, below=0.01cm, font=\small] at (v6) {$v_6$};
   \node[right=0.05cm, below=0.05cm,  font=\small] at (v7) {$v_7$};
   \node[left=0.1cm, font=\small] at (v8) {$v_8$};
    \node[left=0.05cm,font=\small] at (v9) {$v_9$};
   \node[left=0.1cm, above=0.05cm,  font=\small] at (v10) {$v_{10}$};
    \node[above=0.05cm, font=\small] at (v11) {$v_{11}$};
    \node[below=0.05cm,  font=\small] at (v12) {$v_{12}$};
    \node[below=0.05cm, font=\small] at (v13) {$v_{13}$};

 \node[below=0.05cm,  font=\scriptsize] at (v1) {};
\node[below=0.05cm,  font=\scriptsize] at (v2) {};
\node[left=0.05cm,font=\scriptsize] at (v3) {\textcolor{blue}{3}};
\node[above=0.05cm, left=0.05cm,font=\scriptsize] at (v4) {\textcolor{blue}{4}};
\node[left=0.05cm,font=\scriptsize] at (v5) {\textcolor{blue}{4}};
\node[above=0.05cm, font=\scriptsize] at (v6) {\textcolor{blue}{3}};
\node[above=0.05cm, font=\scriptsize] at (v7) {\textcolor{blue}{4}};
\node[right=0.05cm, font=\scriptsize] at (v8) {\textcolor{blue}{4}};
\node[right=0.05cm, font=\scriptsize] at (v9) {\textcolor{blue}{3}};
\node[right=0.05cm,  font=\scriptsize] at (v10) {};
\node[right=0.1cm, font=\scriptsize] at (v11) {};
\node[above=0.1cm, font=\scriptsize] at (v12) {\textcolor{blue}{3}};
\node[left=0.1cm, font=\scriptsize] at (v13) {\textcolor{blue}{3}};

\node[left=0.4cm,below=1cm] at (v6) {\small (b) $f(v)$ of $W_2^2$ (subgraph $H_2$)};
 \end{tikzpicture}
\end{center}
\end{multicols} 
\caption{The numbers indicate the values of $f(v)$ at the corresponding vertices.}
\label{C10-H1-H2}
\end{figure*}
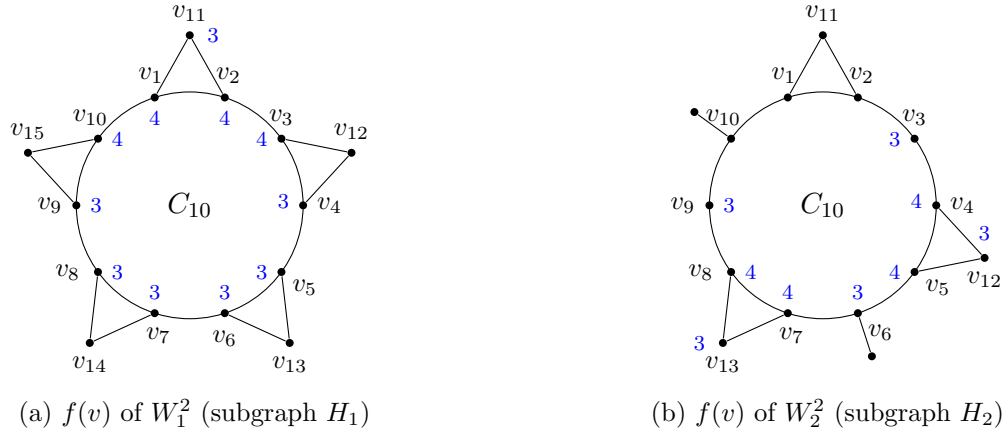
\item \textbf{Case $H_i=H_2$.}

Assume that $G$ contains $H_2$ (see Figure~\ref{H234}). Let $W_2$ be the subgraph of $H_2$ induced by
\[
V(H_2)\setminus\{v_1,v_2,v_{10},v_{11}\}.
\]
Thus,
\[
V(W_2)=\{v_3,v_4,\ldots,v_9,v_{12},v_{13}\}.
\]

For each vertex $v\in V(W_2)$, define
\[
f(v)=6-t_v,
\]
where
\[
t_v=\bigl|N_{G^2}(v)\setminus V(W_2)\bigr|.
\]
Then (see Figure~\ref{C10-H1-H2}(b))
\[
f(v_i)=
\begin{cases}
3,& i\in\{3,6,9,12,13\},\\[2pt]
4,& i\in\{4,5,7,8\}.
\end{cases}
\]

The graph polynomial of $W_2^2$ is
\begin{eqnarray*}
P_{W_2^2}(\bm{x})
&=&\prod_{u\sim v,\;u<v}(x_u-x_v)\\
&=&(x_3-x_4)(x_3-x_5)(x_3-x_{12})
(x_4-x_5)(x_4-x_6)(x_4-x_{12})(x_5-x_6)\\
&&
(x_5-x_7)(x_5-x_{12})(x_6-x_7)(x_6-x_8)
(x_6-x_{12})(x_6-x_{13})(x_7-x_8)\\
&&
(x_7-x_9)(x_7-x_{13})(x_8-x_9)
(x_8-x_{13})(x_9-x_{13}).
\end{eqnarray*}

A direct computation in \textsc{Mathematica} shows that the coefficient of
\[
x_3^2x_4^3x_5^2x_6^2x_7^2x_8^3x_9^2x_{12}^2x_{13}
\]
is $1$. Hence $W_2^2$ is Alon--Tarsi $f$-choosable, a contradiction to  Claim~\ref{key-remark}.
 Therefore $W_2$ cannot occur in $G$, and consequently neither can $H_2$.


\item \textbf{Case $H_i=H_3$.}

Assume that $G$ contains $H_3$ (see Figure~\ref{H234}). For each vertex $v\in V(H_3)$, define
\[
f(v)=6-t_v,
\]
where
\[
t_v=\bigl|N_{G^2}(v)\setminus V(H_3)\bigr|.
\]
Then (see Figure~\ref{H3-H4-color})
\[
f(v_i)=
\begin{cases}
3,& i\in\{6,8,10,11,12\},\\[2pt]
4,& i\in\{1,5,7,9\},\\[2pt]
5,& i\in\{2,4\},\\[2pt]
6,& i=3.
\end{cases}
\]

The graph polynomial of $H_3^2$ is
\begin{eqnarray*}
P_{H_3^2}(\bm{x})
&=&\prod_{u\sim v,\;u<v}(x_u-x_v)\\
&=&(x_1-x_2)(x_1-x_3)(x_1-x_9)(x_1-x_{10})(x_1-x_{11})
(x_2-x_3)(x_2-x_4)\\
&&
(x_2-x_{10})(x_2-x_{11})(x_3-x_4)(x_3-x_5)
(x_3-x_{11})(x_3-x_{12})(x_4-x_5)\\
&&
(x_4-x_6)(x_4-x_{12})(x_5-x_6)(x_5-x_7)
(x_5-x_{12})(x_6-x_7)(x_6-x_8)\\
&&
(x_6-x_{12})(x_7-x_8)(x_7-x_9)
(x_8-x_9)(x_8-x_{10})(x_9-x_{10})(x_{10}-x_{11}).
\end{eqnarray*}

A direct computation in \textsc{Mathematica} shows that the coefficient of
\[
x_1^3x_2^2x_3^2x_4^3x_5^2x_6^2x_7^3x_8^2x_9^3x_{10}^2x_{11}^2x_{12}^2
\]
is $1$. Hence $H_3^2$ is Alon--Tarsi $f$-choosable, a contradiction to  Claim~\ref{key-remark}.
 Therefore $H_3$ cannot occur in $G$.

\begin{figure*}[htbp]
\begin{multicols}{2}
\begin{center}
\begin{tikzpicture}
[u/.style={fill=black, minimum size =3pt,ellipse,inner sep=1pt},node distance=1.5cm,scale=1.5]
\path[use as bounding box] (-1.75,-1.45) rectangle (1.75,1.75);
\node[u] (v1) at (108:1){};
\node[u] (v2) at (72:1){};
\node[u] (v3) at (36:1){};
\node[u] (v4) at (0:1){};
\node[u] (v5) at (324:1){};
\node[u] (v6) at (288:1){};
\node[u] (v7) at (252:1){};
\node[u] (v8) at (216:1){};
\node[u] (v9) at (180:1){};
\node[u] (v10) at (144:1){};
\node[u] (v11) at (90:1.5){};
\node[u] (v12) at (342:1.5){};

\node (C10) at (0, 0){$C_{10}$};

\draw   (0,0) circle[radius=1cm];
  \draw (v1) -- (v11);
  \draw (v2) -- (v11);
  \draw (v4) -- (v12);
  \draw (v5) -- (v12);

\node[u] (w6) at (288:1.4){};
\node[u] (w8) at (216:1.4){};
\node[u] (w10) at (144:1.4){};

  \draw (v6) -- (w6);
  \draw (v8) -- (w8);
  \draw (v10) -- (w10);

       \node[left=0.05cm, above=-0.01cm, font=\scriptsize] at (v1) {$v_1$};
   \node[ right=0.05cm, above=-0.01cm,  font=\scriptsize] at (v2) {$v_2$};
   \node[above=0.05cm, right=0.01cm, font=\scriptsize] at (v3) {$v_3$};
   \node[right=-0.05cm,font=\scriptsize] at (v4) {$v_4$};
    \node[below=0.2cm,right=-0.08cm, font=\scriptsize] at (v5) {$v_5$};
   \node[right=0.3cm, below=-0.01cm, font=\scriptsize] at (v6) {$v_6$};
   \node[right=0.05cm, below=0.01cm,  font=\scriptsize] at (v7) {$v_7$};
   \node[left=-0.05cm, font=\scriptsize] at (v8) {$v_8$};
    \node[left=-0.02cm,font=\scriptsize] at (v9) {$v_9$};
   \node[left=0.1cm, above=0.05cm,  font=\scriptsize] at (v10) {$v_{10}$};
    \node[above=0.02cm, font=\scriptsize] at (v11) {$v_{11}$};
    \node[below=0.03cm,font=\scriptsize] at (v12) {$v_{12}$};

 \node[below=0.05cm,  font=\scriptsize] at (v1) {\textcolor{blue}{4}};
\node[below=0.05cm,  font=\scriptsize] at (v2) {\textcolor{blue}{5}};
\node[left=0.05cm,font=\scriptsize] at (v3) {\textcolor{blue}{6}};
\node[above=0.05cm, left=0.05cm,font=\scriptsize] at (v4) {\textcolor{blue}{5}};
\node[left=0.05cm,font=\scriptsize] at (v5) {\textcolor{blue}{4}};
\node[above=0.05cm, font=\scriptsize] at (v6) {\textcolor{blue}{3}};
\node[above=0.05cm, font=\scriptsize] at (v7) {\textcolor{blue}{4}};
\node[right=0.05cm, font=\scriptsize] at (v8) {\textcolor{blue}{3}};
\node[right=0.05cm, font=\scriptsize] at (v9) {\textcolor{blue}{4}};
\node[right=0.05cm,  font=\scriptsize] at (v10) {\textcolor{blue}{3}};
\node[right=0.1cm, font=\scriptsize] at (v11) {\textcolor{blue}{3}};
\node[above=0.1cm, font=\scriptsize] at (v12) {\textcolor{blue}{3}};
\node[left=0.4cm,below=1cm] at (v6) {\small Subgraph $H_3$};
 \end{tikzpicture}
 \end{center}
\par
\begin{center}
\begin{tikzpicture}
[u/.style={fill=black, minimum size =3pt,ellipse,inner sep=1pt},node distance=1.5cm,scale=1.5]
\path[use as bounding box] (-1.75,-1.45) rectangle (1.75,1.75);
\node[u] (v1) at (108:1){};
\node[u] (v2) at (72:1){};
\node[u] (v3) at (36:1){};
\node[u] (v4) at (0:1){};
\node[u] (v5) at (324:1){};
\node[u] (v6) at (288:1){};
\node[u] (v7) at (252:1){};
\node[u] (v8) at (216:1){};
\node[u] (v9) at (180:1){};
\node[u] (v10) at (144:1){};
\node (C10) at (0, 0){$C_{10}$};

\draw   (0,0) circle[radius=1cm];

 \node[u] (w2) at (72:1.4){};
\node[u] (w4) at (0:1.4){};
\node[u] (w6) at (288:1.4){};
\node[u] (w8) at (216:1.4){};
\node[u] (w10) at (144:1.4){};
  \draw (v2) -- (w2);
  \draw (v4) -- (w4);
  \draw (v6) -- (w6);
  \draw (v8) -- (w8);
  \draw (v10) -- (w10);

    \node[left=0.05cm, above=-0.01cm, font=\scriptsize] at (v1) {$v_1$};
   \node[left=0.1cm, above=0.01cm,  font=\scriptsize] at (v2) {$v_2$};
   \node[above=0.05cm, right=0.01cm, font=\scriptsize] at (v3) {$v_3$};
   \node[right=0.2cm, above=-0.01cm, font=\scriptsize] at (v4) {$v_4$};
    \node[below=0.1cm,right=-0.01cm, font=\scriptsize] at (v5) {$v_5$};
   \node[right=0.25cm, below=-0.03cm, font=\scriptsize] at (v6) {$v_6$};
   \node[right=0.05cm, below=0.01cm,  font=\scriptsize] at (v7) {$v_7$};
   \node[left=-0.01cm, font=\scriptsize] at (v8) {$v_8$};
    \node[left=-0.01cm,font=\scriptsize] at (v9) {$v_9$};
   \node[left=0.1cm, above=0.05cm,  font=\scriptsize] at (v10) {$v_{10}$};

 \node[below=0.05cm,  font=\scriptsize] at (v1) {\textcolor{blue}{4}};
\node[below=0.05cm,  font=\scriptsize] at (v2) {\textcolor{blue}{3}};
\node[left=0.05cm,font=\scriptsize] at (v3) {\textcolor{blue}{4}};
\node[above=-0.05cm, left=0.05cm,font=\scriptsize] at (v4) {\textcolor{blue}{3}};
\node[left=0.05cm,font=\scriptsize] at (v5) {\textcolor{blue}{4}};
\node[above=0.05cm, font=\scriptsize] at (v6) {\textcolor{blue}{3}};
\node[above=0.05cm, font=\scriptsize] at (v7) {\textcolor{blue}{4}};
\node[right=0.02cm, font=\scriptsize] at (v8) {\textcolor{blue}{3}};
\node[right=0.02cm, font=\scriptsize] at (v9) {\textcolor{blue}{4}};
\node[right=0.05cm,  font=\scriptsize] at (v10) {\textcolor{blue}{3}};

\node[left=0.4cm,below=1cm] at (v6) {\small Subgraph $H_4$};
 \end{tikzpicture}
 \end{center}
\end{multicols} 
\caption{The numbers indicate the values of $f(v)$ at the corresponding vertices.}
\label{H3-H4-color}
\end{figure*}

\item \textbf{Case $H_i=H_4$.}

Assume that $G$ contains $H_4$ (see Figure~\ref{H234}). For each vertex $v\in V(H_4)$, define
\[
f(v)=6-t_v,
\]
where
\[
t_v=\bigl|N_{G^2}(v)\setminus V(H_4)\bigr|.
\]
Then (see Figure~\ref{H3-H4-color})
\[
f(v_i)=
\begin{cases}
3,& i\in\{2,4,6,8,10\},\\[2pt]
4,& i\in\{1,3,5,7,9\}.
\end{cases}
\]

The graph polynomial of $H_4^2$ is
\begin{eqnarray*}
P_{H_4^2}(\bm{x})
&=&\prod_{u\sim v,\;u<v}(x_u-x_v)\\
&=&(x_1-x_2)(x_1-x_3)(x_1-x_9)(x_1-x_{10})
(x_2-x_3)(x_2-x_4)(x_2-x_{10})\\
&&
(x_3-x_4)(x_3-x_5)(x_4-x_5)(x_4-x_6)
(x_5-x_6)(x_5-x_7)(x_6-x_7)\\
&&
(x_6-x_8)(x_7-x_8)(x_7-x_9)
(x_8-x_9)(x_8-x_{10})(x_9-x_{10}).
\end{eqnarray*}

A direct computation in \textsc{Mathematica} shows that the coefficient of
\[
x_1^3x_2x_3^2x_4^2x_5^2x_6^2x_7^2x_8^2x_9^2x_{10}^2
\]
is $1$. Hence $H_4^2$ is Alon--Tarsi $f$-choosable, a contradiction to  Claim~\ref{key-remark}. 
 Therefore $H_4$ cannot occur in $G$.
\end{enumerate}

This completes the proof of Claim~\ref{lem-H234}.

\begin{figure}
\centering


\begin{minipage}{0.24\textwidth}\centering
\begin{tikzpicture}[u/.style={fill=black, minimum size =3pt,ellipse,inner sep=1pt},node distance=1.5cm,scale=1.2]
\node[u] (v1) at (80:1){};
\node[u] (v2) at (40:1){};
\node[u] (v3) at (0:1){};
\node[u] (v4) at (320:1){};
\node[u] (v5) at (280:1){};
\node[u] (v6) at (240:1){};
\node[u] (v7) at (200:1){};
\node[u] (v8) at (160:1){};
\node[u] (v9) at (120:1){};
\node (C9) at (0, 0){$C_9$};

\draw   (0,0) circle[radius=1cm];

  \draw (v1) -- (80:1.4);
  \draw (v3) -- (0:1.4);
  \draw (v5) -- (280:1.4);
  \draw (v7) -- (200:1.4);
  \draw (v8) -- (160:1.4);

   \node[left=0.1cm,  below=0.05cm,  font=\small] at (v1) {$v_1$};
   \node[right=0.05cm,font=\small] at (v2) {$v_2$};
   \node[right=0.2cm,above= 0.01cm, font=\small] at (v3) {$v_3$};
   \node[right=0.05cm,font=\small] at (v4) {$v_4$};
    \node[below=0.1cm,right=0.01cm, font=\small] at (v5) {$v_5$};
   \node[below=0.1cm,font=\small] at (v6) {$v_6$};
   \node[above=0.05cm, left=0.1cm, font=\small] at (v7) {$v_7$};
   \node[below=0.1cm, left=0.1cm, font=\small] at (v8) {$v_8$};
    \node[above=0.05cm,font=\small] at (v9) {$v_9$};

  \node[left=0.1cm,below=0.7cm,font=\small] at (v5) {$F_1$};

\end{tikzpicture}

\end{minipage}\hfill
%
\begin{minipage}{0.24\textwidth}\centering
\begin{tikzpicture}[u/.style={fill=black, minimum size =3pt,ellipse,inner sep=1pt},node distance=1.5cm,scale=1.2]
\node[u] (v1) at (80:1){};
\node[u] (v2) at (40:1){};
\node[u] (v3) at (0:1){};
\node[u] (v4) at (320:1){};
\node[u] (v5) at (280:1){};
\node[u] (v6) at (240:1){};
\node[u] (v7) at (200:1){};
\node[u] (v8) at (160:1){};
\node[u] (v9) at (120:1){};
\node[u] (v10) at (220:1.5){};
\node (C9) at (0, 0){$C_9$};

\draw   (0,0) circle[radius=1cm];
  \draw (v6) -- (v10);
  \draw (v7) -- (v10);

  \draw (v1) -- (80:1.4);
  \draw (v3) -- (0:1.4);
  \draw (v5) -- (280:1.4);
  \draw (v8) -- (160:1.4);

   \node[left=0.1cm,  below=0.05cm,  font=\small] at (v1) {$v_1$};
   \node[right=0.05cm,font=\small] at (v2) {$v_2$};
   \node[right=0.2cm,above= 0.01cm, font=\small] at (v3) {$v_3$};
   \node[right=0.05cm,font=\small] at (v4) {$v_4$};
    \node[below=0.1cm,right=0.01cm, font=\small] at (v5) {$v_5$};
   \node[below=0.1cm,font=\small] at (v6) {$v_6$};
   \node[above=0.05cm, left=0.1cm, font=\small] at (v7) {$v_7$};
   \node[below=0.1cm, left=0.1cm, font=\small] at (v8) {$v_8$};
    \node[above=0.05cm,font=\small] at (v9) {$v_9$};

     \node[below=0.1cm,  font=\small] at (v10) {$v_{10}$};

\node[left=0.1cm,below=0.7cm,font=\small] at (v5) {$F_2$};

\end{tikzpicture}

\end{minipage}\hfill
%
\begin{minipage}{0.24\textwidth}\centering
\begin{tikzpicture}[u/.style={fill=black, minimum size =3pt,ellipse,inner sep=1pt},node distance=1.5cm,scale=1.2]
\node[u] (v1) at (80:1){};
\node[u] (v2) at (40:1){};
\node[u] (v3) at (0:1){};
\node[u] (v4) at (320:1){};
\node[u] (v5) at (280:1){};
\node[u] (v6) at (240:1){};
\node[u] (v7) at (200:1){};
\node[u] (v8) at (160:1){};
\node[u] (v9) at (120:1){};
\node[u] (v10) at (220:1.5){};
\node (C9) at (0, 0){$C_9$};

\draw   (0,0) circle[radius=1cm];

  \draw (v1) -- (80:1.4);
  \draw (v3) -- (0:1.4);
  \draw (v4) -- (320:1.4);
  \draw (v8) -- (160:1.4);
  \draw (v6) -- (v10);
  \draw (v7) -- (v10);

   \node[above=0.2cm, left=0.01cm, font=\small] at (v1) {$v_1$};
   \node[above=0.1cm,font=\small] at (v2) {$v_2$};
   \node[right=0.2cm,above= 0.1cm, font=\small] at (v3) {$v_3$};
   \node[below=0.1cm,font=\small] at (v4) {$v_4$};
    \node[below=0.1cm, font=\small] at (v5) {$v_5$};
   \node[below=0.1cm,font=\small] at (v6) {$v_6$};
   \node[above=0.05cm, left=0.1cm, font=\small] at (v7) {$v_7$};
   \node[below=0.1cm, left=0.1cm, font=\small] at (v8) {$v_8$};
    \node[above=0.1cm, left=0.01cm, font=\small] at (v9) {$v_9$};
     \node[below=0.1cm,font=\small] at (v10) {$v_{10}$};

\node[left=0.1cm,below=0.7cm,font=\small] at (v5) {$F_3$};

\end{tikzpicture}

\end{minipage}\hfill
%
\begin{minipage}{0.24\textwidth}\centering
\begin{tikzpicture}[u/.style={fill=black, minimum size =3pt,ellipse,inner sep=1pt},node distance=1.5cm,scale=1.2]
\node[u] (v1) at (80:1){};
\node[u] (v2) at (40:1){};
\node[u] (v3) at (0:1){};
\node[u] (v4) at (320:1){};
\node[u] (v5) at (280:1){};
\node[u] (v6) at (240:1){};
\node[u] (v7) at (200:1){};
\node[u] (v8) at (160:1){};
\node[u] (v9) at (120:1){};
\node[u] (v10) at (60:1.5){};
\node[u] (v11) at (260:1.5){};
\node (C9) at (0, 0){$C_9$};

\draw   (0,0) circle[radius=1cm];
  \draw (v1) -- (v10);
  \draw (v2) -- (v10);
  \draw (v5) -- (v11);
  \draw (v6) -- (v11);


  \draw (v4) -- (320:1.4);
  \draw (v7) -- (200:1.4);
 \draw (v9) -- (120:1.4);

   \node[above=0.05cm, font=\small] at (v1) {$v_1$};
   \node[right=0.05cm,font=\small] at (v2) {$v_2$};
   \node[right=0.05cm, font=\small] at (v3) {$v_3$};
   \node[right=0.05cm,font=\small] at (v4) {$v_4$};
    \node[right=0.05cm, below=0.1cm, font=\small] at (v5) {$v_5$};
   \node[below=0.05cm,left=0.05cm, font=\small] at (v6) {$v_6$};
   \node[above=0.1cm, left=0.05cm, font=\small] at (v7) {$v_7$};
   \node[left=0.05cm, font=\small] at (v8) {$v_8$};
    \node[left=0.05cm, font=\small] at (v9) {$v_9$};
   \node[above=0.05cm,  font=\small] at (v10) {$v_{10}$};
    \node[left=0.05cm,  font=\small] at (v11) {$v_{11}$};

\node[left=0.1cm,below=0.7cm,font=\small] at (v5) {$F_4$};

\end{tikzpicture}

\end{minipage}

\vspace{0.5cm}


\begin{minipage}{0.24\textwidth}\centering
\begin{tikzpicture}[u/.style={fill=black, minimum size =3pt,ellipse,inner sep=1pt},node distance=1.5cm,scale=1.2]
\node[u] (v1) at (80:1){};
\node[u] (v2) at (40:1){};
\node[u] (v3) at (0:1){};
\node[u] (v4) at (320:1){};
\node[u] (v5) at (280:1){};
\node[u] (v6) at (240:1){};
\node[u] (v7) at (200:1){};
\node[u] (v8) at (160:1){};
\node[u] (v9) at (120:1){};
\node[u] (v10) at (60:1.5){};
\node[u] (v11) at (340:1.5){};
\node (C9) at (0, 0){$C_9$};

\draw   (0,0) circle[radius=1cm];
  \draw (v1) -- (v10);
  \draw (v2) -- (v10);
  \draw (v3) -- (v11);
  \draw (v4) -- (v11);

  \draw (v5) -- (280:1.4);
  \draw (v7) -- (200:1.4);
 \draw (v9) -- (120:1.4);

   \node[above=0.1cm, font=\small] at (v1) {$v_1$};
   \node[right=0.1cm,font=\small] at (v2) {$v_2$};
   \node[right=0.05cm, font=\small] at (v3) {$v_3$};
   \node[right=0.05cm, below=0.1cm,font=\small] at (v4) {$v_4$};
    \node[below=0.2cm,right=0.01cm, font=\small] at (v5) {$v_5$};
   \node[below=0.05cm,font=\small] at (v6) {$v_6$};
   \node[above=0.05cm, left=0.05cm, font=\small] at (v7) {$v_7$};
   \node[left=0.05cm, font=\small] at (v8) {$v_8$};
    \node[left=0.1cm,font=\small] at (v9) {$v_9$};
   \node[above=0.1cm,  font=\small] at (v10) {$v_{10}$};
    \node[above=0.1cm, right=0.05cm, font=\small] at (v11) {$v_{11}$};

\node[left=0.1cm,below=0.7cm,font=\small] at (v5) {$F_5$};

\end{tikzpicture}

\end{minipage}\hfill
%
\begin{minipage}{0.24\textwidth}\centering
\begin{tikzpicture}[u/.style={fill=black, minimum size =3pt,ellipse,inner sep=1pt},node distance=1.5cm,scale=1.2]
\node[u] (v1) at (80:1){};
\node[u] (v2) at (40:1){};
\node[u] (v3) at (0:1){};
\node[u] (v4) at (320:1){};
\node[u] (v5) at (280:1){};
\node[u] (v6) at (240:1){};
\node[u] (v7) at (200:1){};
\node[u] (v8) at (160:1){};
\node[u] (v9) at (120:1){};
\node[u] (v10) at (60:1.5){};
\node[u] (v11) at (220:1.5){};
\node (C9) at (0, 0){$C_9$};

\draw   (0,0) circle[radius=1cm];
  \draw (v1) -- (v10);
  \draw (v2) -- (v10);
  \draw (v6) -- (v11);
  \draw (v7) -- (v11);

  \draw (v4) -- (320:1.4);
  \draw (v5) -- (280:1.4);
 \draw (v9) -- (120:1.4);

   \node[above=0.05cm, font=\small] at (v1) {$v_1$};
   \node[right=0.05cm,font=\small] at (v2) {$v_2$};
   \node[right=0.05cm, font=\small] at (v3) {$v_3$};
   \node[right=0.05cm,font=\small] at (v4) {$v_4$};
    \node[below=0.1cm,right=0.01cm, font=\small] at (v5) {$v_5$};
   \node[below=0.05cm,font=\small] at (v6) {$v_6$};
   \node[above=0.05cm, left=0.05cm, font=\small] at (v7) {$v_7$};
   \node[left=0.05cm, font=\small] at (v8) {$v_8$};
    \node[left=0.05cm,font=\small] at (v9) {$v_9$};
   \node[above=0.05cm,  font=\small] at (v10) {$v_{10}$};
    \node[below=0.05cm, font=\small] at (v11) {$v_{11}$};

  \node[left=0.1cm,below=0.7cm, font=\small] at (v5) {$F_6$};

\end{tikzpicture}

\end{minipage}\hfill
%
\begin{minipage}{0.24\textwidth}\centering
\begin{tikzpicture}[u/.style={fill=black, minimum size =3pt,ellipse,inner sep=1pt},node distance=1.5cm,scale=1.2]
\node[u] (v1) at (80:1){};
\node[u] (v2) at (40:1){};
\node[u] (v3) at (0:1){};
\node[u] (v4) at (320:1){};
\node[u] (v5) at (280:1){};
\node[u] (v6) at (240:1){};
\node[u] (v7) at (200:1){};
\node[u] (v8) at (160:1){};
\node[u] (v9) at (120:1){};
\node[u] (v10) at (60:1.5){};
\node[u] (v11) at (180:1.5){};
\node (C9) at (0, 0){$C_9$};

\draw   (0,0) circle[radius=1cm];
  \draw (v1) -- (v10);
  \draw (v2) -- (v10);
  \draw (v7) -- (v11);
  \draw (v8) -- (v11);

  \draw (v4) -- (320:1.4);
  \draw (v5) -- (280:1.4);
 \draw (v9) -- (120:1.4);

   \node[above=0.05cm, font=\small] at (v1) {$v_1$};
   \node[right=0.05cm, font=\small] at (v2) {$v_2$};
   \node[right=0.05cm, font=\small] at (v3) {$v_3$};
   \node[right=0.05cm, font=\small] at (v4) {$v_4$};
    \node[below=0.1cm, right=0.01cm, font=\small] at (v5) {$v_5$};
   \node[below=0.05cm, font=\small] at (v6) {$v_6$};
   \node[below=0.3cm, left=0.01cm, font=\small] at (v7) {$v_7$};
   \node[above=0.05cm, left=0.05cm, font=\small] at (v8) {$v_8$};
    \node[left=0.05cm, font=\small] at (v9) {$v_9$};
   \node[above=0.05cm,  font=\small] at (v10) {$v_{10}$};
    \node[below=0.05cm, font=\small] at (v11) {$v_{11}$};

\node[left=0.1cm,below=0.7cm, font=\small]at (v5) {$F_7$};
\end{tikzpicture}

\end{minipage}\hfill
%
\begin{minipage}{0.24\textwidth}\centering
\begin{tikzpicture}[u/.style={fill=black, minimum size =3pt,ellipse,inner sep=1pt},node distance=1.5cm,scale=1.2]
\node[u] (v1) at (80:1){};
\node[u] (v2) at (40:1){};
\node[u] (v3) at (0:1){};
\node[u] (v4) at (320:1){};
\node[u] (v5) at (280:1){};
\node[u] (v6) at (240:1){};
\node[u] (v7) at (200:1){};
\node[u] (v8) at (160:1){};
\node[u] (v9) at (120:1){};
\node[u] (v10) at (60:1.5){};
\node[u] (v11) at (300:1.5){};
\node (C9) at (0, 0){$C_9$};

\draw   (0,0) circle[radius=1cm];
  \draw (v1) -- (v10);
  \draw (v2) -- (v10);
  \draw (v4) -- (v11);
  \draw (v5) -- (v11);


  \draw (v6) -- (240:1.4);
  \draw (v7) -- (200:1.4);
 \draw (v9) -- (120:1.4);

   \node[above=0.05cm, font=\small] at (v1) {$v_1$};
   \node[right=0.05cm,font=\small] at (v2) {$v_2$};
   \node[right=0.05cm, font=\small] at (v3) {$v_3$};
   \node[right=0.05cm, font=\small] at (v4) {$v_4$};
    \node[below=0.05cm,  font=\small] at (v5) {$v_5$};
   \node[left=0.05cm,font=\small] at (v6) {$v_6$};
   \node[above=0.2cm, left=0.05cm, font=\small] at (v7) {$v_7$};
   \node[left=0.05cm, font=\small] at (v8) {$v_8$};
    \node[left=0.05cm,font=\small] at (v9) {$v_9$};
   \node[above=0.05cm,  font=\small] at (v10) {$v_{10}$};
    \node[below=0.05cm, font=\small] at (v11) {$v_{11}$};

\node[left=0.1cm,below=0.7cm, font=\small] at (v5) {$F_8$};

\end{tikzpicture}

\end{minipage}

\vspace{0.5cm}

\begin{minipage}{0.24\textwidth}\centering
\begin{tikzpicture}[u/.style={fill=black, minimum size =3pt,ellipse,inner sep=1pt},node distance=1.5cm,scale=1.2]
\node[u] (v1) at (80:1){};
\node[u] (v2) at (40:1){};
\node[u] (v3) at (0:1){};
\node[u] (v4) at (320:1){};
\node[u] (v5) at (280:1){};
\node[u] (v6) at (240:1){};
\node[u] (v7) at (200:1){};
\node[u] (v8) at (160:1){};
\node[u] (v9) at (120:1){};
\node[u] (v10) at (60:1.5){};
\node[u] (v11) at (300:1.5){};
\node[u] (v12) at (180:1.5){};
\node (C9) at (0, 0){$C_9$};

\draw   (0,0) circle[radius=1cm];
  \draw (v1) -- (v10);
  \draw (v2) -- (v10);
  \draw (v4) -- (v11);
  \draw (v5) -- (v11);

  \draw (v7) -- (v12);
  \draw (v8) -- (v12);

  \draw (v6) -- (240:1.4);
 \draw (v9) -- (120:1.4);

   \node[above=0.05cm, font=\small] at (v1) {$v_1$};
   \node[right=0.05cm,font=\small] at (v2) {$v_2$};
   \node[right=0.05cm, font=\small] at (v3) {$v_3$};
   \node[right=0.05cm,font=\small] at (v4) {$v_4$};
    \node[below=0.05cm,  font=\small] at (v5) {$v_5$};
   \node[left=0.05cm,font=\small] at (v6) {$v_6$};
   \node[below=0.3cm, left=0.05cm, font=\small] at (v7) {$v_7$};
   \node[left=0.05cm, font=\small] at (v8) {$v_8$};
    \node[left=0.05cm,font=\small] at (v9) {$v_9$};
   \node[above=0.05cm,  font=\small] at (v10) {$v_{10}$};
    \node[below=0.05cm, font=\small] at (v11) {$v_{11}$};
    \node[below=0.05cm,  font=\small] at (v12) {$v_{12}$};

\node[left=0.1cm,below=0.7cm, font=\small] at (v5) { $F_9$};
\end{tikzpicture}

\end{minipage}\hfill
%
\begin{minipage}{0.24\textwidth}\centering
\begin{tikzpicture}[u/.style={fill=black, minimum size =3pt,ellipse,inner sep=1pt},node distance=1.5cm,scale=1.2]
\node[u] (v1) at (80:1){};
\node[u] (v2) at (40:1){};
\node[u] (v3) at (0:1){};
\node[u] (v4) at (320:1){};
\node[u] (v5) at (280:1){};
\node[u] (v6) at (240:1){};
\node[u] (v7) at (200:1){};
\node[u] (v8) at (160:1){};
\node[u] (v9) at (120:1){};
\node[u] (v10) at (60:1.5){};
\node[u] (v11) at (340:1.5){};
\node[u] (v12) at (220:1.5){};
\node (C9) at (0, 0){$C_9$};

\draw   (0,0) circle[radius=1cm];
  \draw (v1) -- (v10);
  \draw (v2) -- (v10);
  \draw (v3) -- (v11);
  \draw (v4) -- (v11);
  \draw (v6) -- (v12);
  \draw (v7) -- (v12);

  \draw (v5) -- (280:1.4);
 \draw (v9) -- (120:1.4);

   \node[above=0.05cm, font=\small] at (v1) {$v_1$};
   \node[right=0.05cm,font=\small] at (v2) {$v_2$};
   \node[right=0.05cm, font=\small] at (v3) {$v_3$};
   \node[right=0.1cm, below=0.05cm, font=\small] at (v4) {$v_4$};
    \node[below=0.2cm,right=0.01cm, font=\small] at (v5) {$v_5$};
   \node[below=0.05cm,font=\small] at (v6) {$v_6$};
   \node[above=0.05cm, left=0.05cm, font=\small] at (v7) {$v_7$};
   \node[left=0.05cm, font=\small] at (v8) {$v_8$};
    \node[left=0.05cm,font=\small] at (v9) {$v_9$};
   \node[above=0.05cm,  font=\small] at (v10) {$v_{10}$};
    \node[below=0.05cm,  font=\small] at (v11) {$v_{11}$};
     \node[above=0.05cm, left=0.05cm, font=\small] at (v12) {$v_{12}$};
  \node[left=0.1cm,below=0.7cm,font=\small] at (v5) { $F_{10}$};
\end{tikzpicture}

\end{minipage}\hfill
%
\begin{minipage}{0.24\textwidth}\centering
\begin{tikzpicture}[u/.style={fill=black, minimum size =3pt,ellipse,inner sep=1pt},node distance=1.5cm,scale=1.2]
\node[u] (v1) at (80:1){};
\node[u] (v2) at (40:1){};
\node[u] (v3) at (0:1){};
\node[u] (v4) at (320:1){};
\node[u] (v5) at (280:1){};
\node[u] (v6) at (240:1){};
\node[u] (v7) at (200:1){};
\node[u] (v8) at (160:1){};
\node[u] (v9) at (120:1){};
\node[u] (v10) at (60:1.5){};
\node[u] (v11) at (340:1.5){};
\node[u] (v12) at (260:1.5){};
\node (C9) at (0, 0){$C_9$};

\draw   (0,0) circle[radius=1cm];
  \draw (v1) -- (v10);
  \draw (v2) -- (v10);
  \draw (v3) -- (v11);
  \draw (v4) -- (v11);

  \draw (v5) -- (v12);
  \draw (v6) -- (v12);

  \draw (v7) -- (200:1.4);
 \draw (v9) -- (120:1.4);

   \node[above=0.05cm, font=\small] at (v1) {$v_1$};
   \node[right=0.05cm,font=\small] at (v2) {$v_2$};
   \node[right=0.05cm, font=\small] at (v3) {$v_3$};
   \node[right=0.05cm, below=0.1cm, font=\small] at (v4) {$v_4$};
    \node[right=0.05cm, below=0.1cm, font=\small] at (v5) {$v_5$};
   \node[below=0.05cm, left=0.05cm, font=\small] at (v6) {$v_6$};
   \node[above=0.05cm, left=0.05cm, font=\small] at (v7) {$v_7$};
   \node[left=0.05cm, font=\small] at (v8) {$v_8$};
    \node[left=0.05cm,font=\small] at (v9) {$v_9$};
   \node[above=0.05cm,  font=\small] at (v10) {$v_{10}$};
    \node[below=0.05cm, font=\small] at (v11) {$v_{11}$};
    \node[left=0.05cm,  font=\small] at (v12) {$v_{12}$};

\node[left=0.1cm,below=0.7cm, font=\small] at (v5) {$F_{11}$};
\end{tikzpicture}

\end{minipage}\hfill
%
\begin{minipage}{0.24\textwidth}\centering
\begin{tikzpicture}[u/.style={fill=black, minimum size =3pt,ellipse,inner sep=1pt},node distance=1.5cm,scale=1.2]
\node[u] (v1) at (80:1){};
\node[u] (v2) at (40:1){};
\node[u] (v3) at (0:1){};
\node[u] (v4) at (320:1){};
\node[u] (v5) at (280:1){};
\node[u] (v6) at (240:1){};
\node[u] (v7) at (200:1){};
\node[u] (v8) at (160:1){};
\node[u] (v9) at (120:1){};
\node[u] (v10) at (60:1.5){};
\node[u] (v11) at (340:1.5){};
\node[u] (v12) at (260:1.5){};
\node[u] (v13) at (180:1.5){};
\node (C9) at (0, 0){$C_9$};

\draw   (0,0) circle[radius=1cm];
  \draw (v1) -- (v10);
  \draw (v2) -- (v10);
  \draw (v3) -- (v11);
  \draw (v4) -- (v11);

  \draw (v5) -- (v12);
  \draw (v6) -- (v12);
  \draw (v7) -- (v13);
  \draw (v8) -- (v13);

 \draw (v9) -- (120:1.4);

   \node[above=0.05cm, font=\small] at (v1) {$v_1$};
   \node[right=0.05cm, font=\small] at (v2) {$v_2$};
   \node[right=0.05cm, font=\small] at (v3) {$v_3$};
   \node[below=0.05cm,font=\small] at (v4) {$v_4$};
    \node[below=0.05cm,  font=\small] at (v5) {$v_5$};
   \node[left=0.05cm, font=\small] at (v6) {$v_6$};
   \node[left=0.1cm, below=0.05cm, font=\small] at (v7) {$v_7$};
   \node[left=0.05cm, font=\small] at (v8) {$v_8$};
    \node[left=0.05cm, font=\small] at (v9) {$v_9$};
   \node[above=0.05cm,  font=\small] at (v10) {$v_{10}$};
    \node[above=0.05cm, right=0.05cm, font=\small] at (v11) {$v_{11}$};
    \node[left=0.05cm,  font=\small] at (v12) {$v_{12}$};
     \node[below=0.05cm, font=\small] at (v13) {$v_{13}$};

\node[left=0.1cm,below=0.7cm, font=\small] at (v5) {$F_{12}$};
\end{tikzpicture}
\end{minipage}
\caption{The twelve reducible configurations on $9$-cycles.}
\label{W12-subgraph}
\end{figure}
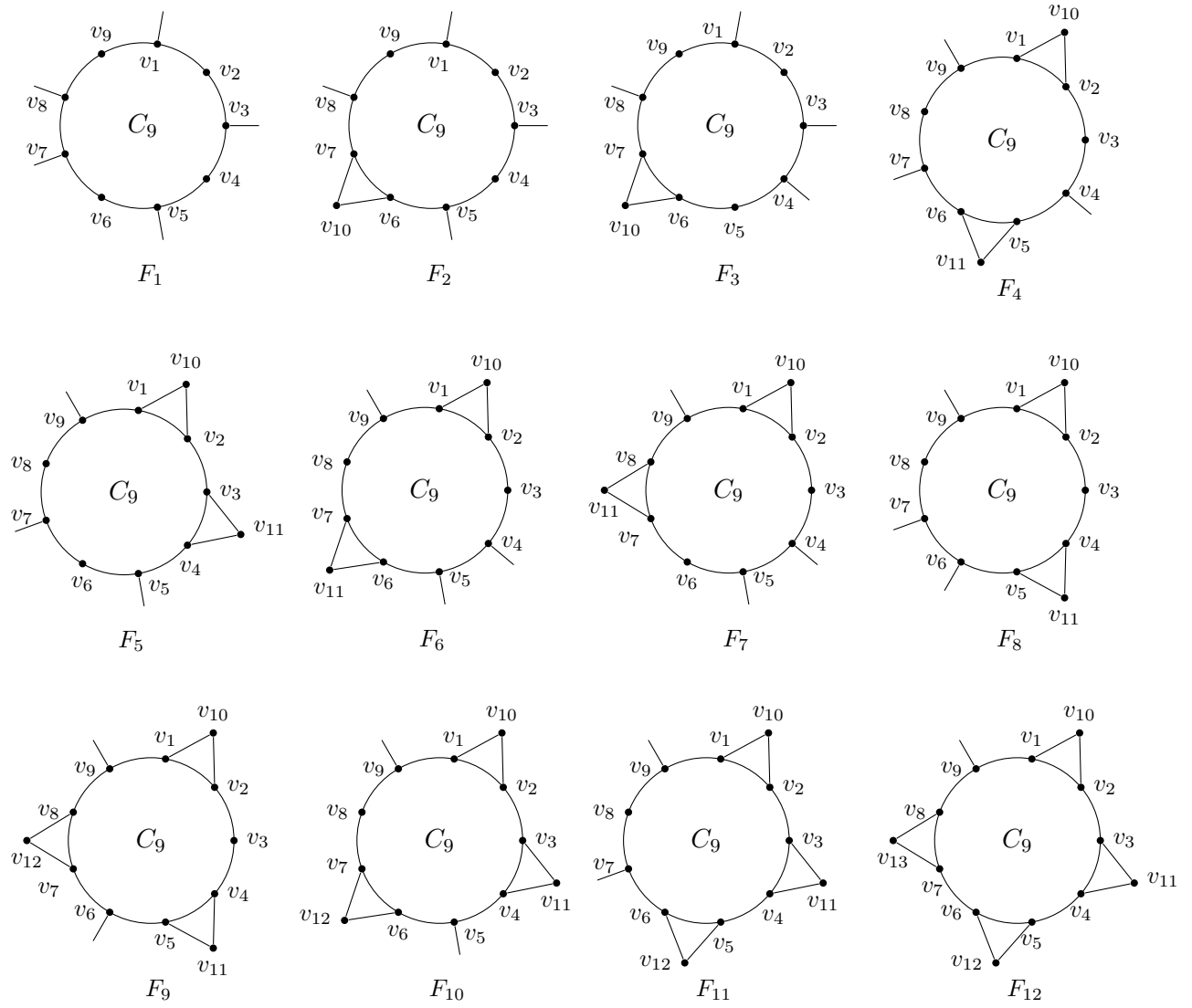
\subsection{Proof of Claim~\ref{claim-F-one}: Reducible configurations on $9$-cycles}

We  prove Claim~\ref{claim-F-one} by showing that none of the configurations $F_i$s can occur in $G$ where $i \in \{1,2,\dots, 12\}$.

\noindent\textbf{\textit{Proof of Claim~\ref{claim-F-one}.}}
For each $i\in\{2,3,5,9,10,12\}$, Claim~24 of~\cite{KLuo25}
already established, via the graph polynomial method, that $F_i^2$ is
Alon--Tarsi $f$-choosable. Therefore, none of these configurations can occur
in $G$. It remains to consider $i\in\{1,4,6,7,8,11\}$.

Suppose, to the contrary, that $G$ contains one of the configurations $F_i$, where $i\in\{1,4,6,7,8,11\}$.

We treat each configuration $F_i$, where
$i\in\{1,\ldots,12\}$, in the same way.
Assume that $G$ contains $F_i$ as a subgraph, and define
\[
f(v)=6-t_v,
\qquad
t_v=\bigl|N_{G^2}(v)\setminus V(F_i)\bigr|
\]
for every $v\in V(F_i)$.
If $F_i^2$ is Alon--Tarsi $f$-choosable, then we obtain a contradiction to Claim~\ref{key-remark}. 
Hence $F_i$ cannot occur in $G$.

Since the arguments are identical, for each remaining configuration we record
only the function $f$, the graph polynomial $P_{F_i^2}(\bm{x})$, and a monomial
with nonzero coefficient.

\medskip
\begin{enumerate}[(a)]

\item \textbf{Case $F_i=F_1$.}

The function $f$ is given by (see Figure~\ref{F1-color})
\[
f(v_i)=
\begin{cases}
2, & i\in\{7,8\},\\
3, & i\in\{1,3,5\},\\
4, & i\in\{2,4,6,9\}.
\end{cases}
\]

The graph polynomial of $F_1^2$ is
\begin{eqnarray*}
P_{F_1^2}(\bm{x})
&=& \prod_{u\sim v,\;u<v}(x_u-x_v)\\
&=& (x_1-x_2)(x_1-x_3)(x_1-x_8)(x_1-x_9)
(x_2-x_3)(x_2-x_4)(x_2-x_9)\\
&& (x_3-x_4)(x_3-x_5)(x_4-x_5)(x_4-x_6)
(x_5-x_6)(x_5-x_7)\\
&& (x_6-x_7)(x_6-x_8)(x_7-x_8)(x_7-x_9)(x_8-x_9).
\end{eqnarray*}
A direct computation in \textsc{Mathematica} shows that the coefficient of
\[
x_1^2x_2^2x_3x_4^3x_5^2x_6^3x_7x_8x_9^3
\]
is $1$.

\begin{figure}
\centering


\begin{minipage}{0.24\textwidth}\centering
\raisebox{-34ex}{\begin{tikzpicture}[u/.style={fill=black, minimum size =3pt,ellipse,inner sep=1pt},node distance=1.5cm,scale=1.5] 
\node[u] (v1) at (80:1){};
\node[u] (v2) at (40:1){};
\node[u] (v3) at (0:1){};
\node[u] (v4) at (320:1){};
\node[u] (v5) at (280:1){};
\node[u] (v6) at (240:1){};
\node[u] (v7) at (200:1){};
\node[u] (v8) at (160:1){};
\node[u] (v9) at (120:1){};
\node (C9) at (0, 0){$C_9$};

\draw   (0,0) circle[radius=1cm];

\node[u] (w1) at (80:1.4){};
\node[u] (w3) at (0:1.4){};
\node[u] (w5) at (280:1.4){};
\node[u] (w7) at (200:1.4){};
\node[u] (w8) at (160:1.4){};
  \draw (v1) -- (w1);
  \draw (v3) -- (w3);
  \draw (v5) -- (w5);
  \draw (v7) -- (w7);
  \draw (v8) -- (w8);

   \node[above=0.2cm, left=0.05cm, font=\small] at (v1) {$v_1$};
   \node[above=0.1cm,font=\small] at (v2) {$v_2$};
   \node[right=0.2cm,above= 0.1cm, font=\small] at (v3) {$v_3$};
   \node[below=0.1cm,font=\small] at (v4) {$v_4$};
    \node[below=0.2cm,left=0.01cm, font=\small] at (v5) {$v_5$};
   \node[below=0.1cm,font=\small] at (v6) {$v_6$};
   \node[above=0.05cm, left=0.1cm, font=\small] at (v7) {$v_7$};
   \node[below=0.1cm, left=0.1cm, font=\small] at (v8) {$v_8$};
    \node[above=0.1cm,font=\small] at (v9) {$v_9$};

 \node[below=0.1cm,  font=\scriptsize] at (v1) {\textcolor{blue}{3}};
\node[below=0.1cm, left=0.1cm, font=\scriptsize] at (v2) {\textcolor{blue}{4}};
\node[left=0.1cm,font=\scriptsize] at (v3) {\textcolor{blue}{3}};
\node[above=0.1cm, left=0.1cm,font=\scriptsize] at (v4) {\textcolor{blue}{4}};
\node[above=0.05cm,font=\scriptsize] at (v5) {\textcolor{blue}{3}};
\node[above=0.1cm, font=\scriptsize] at (v6) {\textcolor{blue}{4}};
\node[right=0.1cm, font=\scriptsize] at (v7) {\textcolor{blue}{2}};
\node[right=0.1cm, font=\scriptsize] at (v8) {\textcolor{blue}{2}};
\node[right=0.1cm, below=0.1cm, font=\scriptsize] at (v9) {\textcolor{blue}{4}};
 \node[left=0.1cm,below=1cm,font=\small] at (v5) {(a) $F_1$};

\end{tikzpicture}}

\end{minipage}\hfill
%
\begin{minipage}{0.24\textwidth}\centering
\begin{tikzpicture}[u/.style={fill=black, minimum size =3pt,ellipse,inner sep=1pt},node distance=1.5cm,scale=1.5]
\node[u] (v1) at (80:1){};
\node[u] (v2) at (40:1){};
\node[u] (v3) at (0:1){};
\node[u] (v4) at (320:1){};
\node[u] (v5) at (280:1){};
\node[u] (v6) at (240:1){};
\node[u] (v7) at (200:1){};
\node[u] (v8) at (160:1){};
\node[u] (v9) at (120:1){};
\node[u] (v10) at (60:1.5){};
\node[u] (v11) at (260:1.5){};
\node (C9) at (0, 0){$C_9$};

\draw   (0,0) circle[radius=1cm];
  \draw (v1) -- (v10);
  \draw (v2) -- (v10);
  \draw (v5) -- (v11);
  \draw (v6) -- (v11);

\node[u] (w4) at (320:1.4){};
\node[u] (w7) at (200:1.4){};
\node[u] (w9) at (120:1.4){};
  \draw (v4) -- (w4);
  \draw (v7) -- (w7);
 \draw (v9) -- (w9);

   \node[above=0.05cm, font=\small] at (v1) {$v_1$};
   \node[right=0.05cm,font=\small] at (v2) {$v_2$};
   \node[right=0.05cm, font=\small] at (v3) {$v_3$};
   \node[right=0.1cm,font=\small] at (v4) {$v_4$};
    \node[right=0.05cm, below=0.1cm, font=\small] at (v5) {$v_5$};
   \node[below=0.1cm,left=0.05cm, font=\small] at (v6) {$v_6$};
   \node[above=0.1cm, left=0.05cm, font=\small] at (v7) {$v_7$};
   \node[left=0.05cm, font=\small] at (v8) {$v_8$};
    \node[left=0.1cm,font=\small] at (v9) {$v_9$};
   \node[above=0.1cm,  font=\small] at (v10) {$v_{10}$};
    \node[left=0.1cm,  font=\small] at (v11) {$v_{11}$};

 \node[below=0.05cm,  font=\scriptsize] at (v1) {\textcolor{blue}{4}};
\node[below=0.05cm, left=0.1cm, font=\scriptsize] at (v2) {\textcolor{blue}{5}};
\node[left=0.05cm,font=\scriptsize] at (v3) {\textcolor{blue}{5}};
\node[above=0.05cm, left=0.1cm,font=\scriptsize] at (v4) {\textcolor{blue}{3}};
\node[above=0.05cm,font=\scriptsize] at (v5) {\textcolor{blue}{4}};
\node[above=0.05cm, font=\scriptsize] at (v6) {\textcolor{blue}{4}};
\node[right=0.05cm, font=\scriptsize] at (v7) {\textcolor{blue}{3}};
\node[right=0.05cm, font=\scriptsize] at (v8) {\textcolor{blue}{4}};
\node[below=0.05cm, font=\scriptsize] at (v9) {\textcolor{blue}{3}};
\node[right=0.1cm, font=\scriptsize] at (v10) {\textcolor{blue}{3}};
\node[right=0.05cm, font=\scriptsize] at (v11) {\textcolor{blue}{3}};

\node[left=0.1cm,below=1.0cm,font=\small] at (v5) {(b) $F_4$};

\end{tikzpicture}

\end{minipage}\hfill
%
\begin{minipage}{0.24\textwidth}\centering
\begin{tikzpicture}[u/.style={fill=black, minimum size =3pt,ellipse,inner sep=1pt},node distance=1.5cm,scale=1.5]
\node[u] (v1) at (80:1){};
\node[u] (v2) at (40:1){};
\node[u] (v3) at (0:1){};
\node[u] (v4) at (320:1){};
\node[u] (v5) at (280:1){};
\node[u] (v6) at (240:1){};
\node[u] (v7) at (200:1){};
\node[u] (v8) at (160:1){};
\node[u] (v9) at (120:1){};
\node[u] (v10) at (60:1.5){};
\node[u] (v11) at (220:1.5){};
\node (C9) at (0, 0){$C_9$};

\draw   (0,0) circle[radius=1cm];
  \draw (v1) -- (v10);
  \draw (v2) -- (v10);
  \draw (v6) -- (v11);
  \draw (v7) -- (v11);

\node[u] (w4) at (320:1.4){};
\node[u] (w5) at (280:1.4){};
\node[u] (w9) at (120:1.4){};
  \draw (v4) -- (w4);
  \draw (v5) -- (w5);
 \draw (v9) -- (w9);

   \node[above=0.1cm, font=\small] at (v1) {$v_1$};
   \node[right=0.1cm,font=\small] at (v2) {$v_2$};
   \node[right=0.1cm, font=\small] at (v3) {$v_3$};
   \node[right=0.1cm,font=\small] at (v4) {$v_4$};
    \node[below=0.2cm,right=0.01cm, font=\small] at (v5) {$v_5$};
   \node[below=0.1cm,font=\small] at (v6) {$v_6$};
   \node[above=0.05cm, left=0.05cm, font=\small] at (v7) {$v_7$};
   \node[left=0.1cm, font=\small] at (v8) {$v_8$};
    \node[left=0.1cm,font=\small] at (v9) {$v_9$};
   \node[above=0.1cm,  font=\small] at (v10) {$v_{10}$};
    \node[below=0.1cm, font=\small] at (v11) {$v_{11}$};

 \node[below=0.05cm,  font=\scriptsize] at (v1) {\textcolor{blue}{4}};
\node[below=0.05cm, left=0.1cm, font=\scriptsize] at (v2) {\textcolor{blue}{5}};
\node[left=0.05cm,font=\scriptsize] at (v3) {\textcolor{blue}{5}};
\node[above=0.05cm, left=0.1cm,font=\scriptsize] at (v4) {\textcolor{blue}{2}};
\node[above=0.05cm,font=\scriptsize] at (v5) {\textcolor{blue}{2}};
\node[above=0.05cm, font=\scriptsize] at (v6) {\textcolor{blue}{4}};
\node[right=0.05cm, font=\scriptsize] at (v7) {\textcolor{blue}{5}};
\node[right=0.05cm, font=\scriptsize] at (v8) {\textcolor{blue}{5}};
\node[below=0.05cm, font=\scriptsize] at (v9) {\textcolor{blue}{3}};
\node[ below=0.1cm, right=0.1cm, font=\scriptsize] at (v10) {\textcolor{blue}{3}};
\node[left=0.1cm, font=\scriptsize] at (v11) {\textcolor{blue}{3}};

\node[left=0.1cm,below=1.0cm,font=\small] at (v5) {(c) $F_6$};

\end{tikzpicture}

\end{minipage}\hfill
\caption{The numbers indicate the values of $f(v)$ at the corresponding vertices.}
\label{F1-color}
\end{figure}
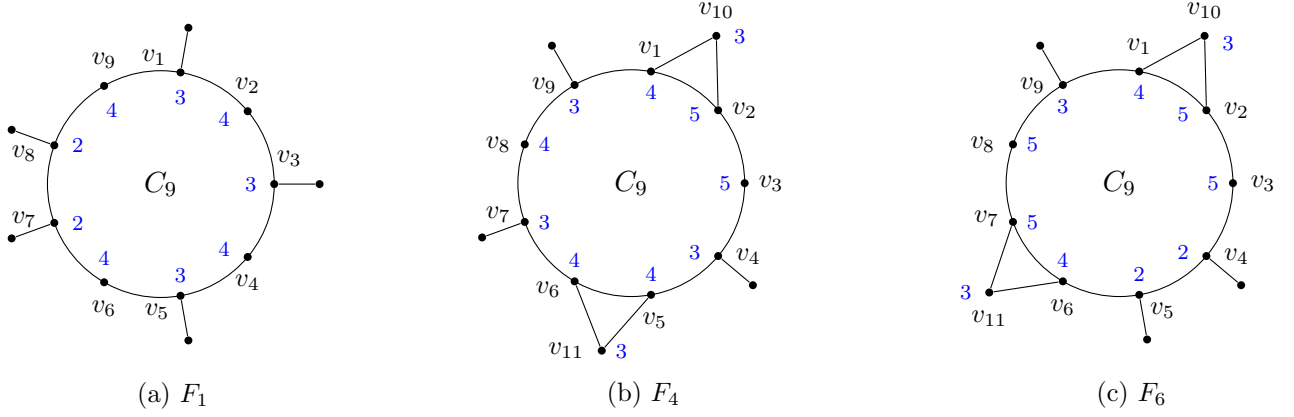


\item \textbf{Case $F_i=F_4$.}

The function $f$ is given by (see Figure~\ref{F1-color})
\[
f(v_i)=
\begin{cases}
3, & i\in\{4,7,9,10,11\},\\
4, & i\in\{1,5,6,8\},\\
5, & i\in\{2,3\}.
\end{cases}
\]

The graph polynomial of $F_4^2$ is
\begin{eqnarray*}
P_{F_4^2}(\bm{x})
&=& \prod_{u\sim v,\;u<v}(x_u-x_v)\\
&=& (x_1-x_2)(x_1-x_3)(x_1-x_8)(x_1-x_9)(x_1-x_{10})
(x_2-x_3)(x_2-x_4)\\
&& (x_2-x_9)(x_2-x_{10})(x_3-x_4)(x_3-x_5)(x_3-x_{10})
(x_4-x_5)\\
&& (x_4-x_6)(x_4-x_{11})(x_5-x_6)(x_5-x_7)(x_5-x_{11})
(x_6-x_7)\\
&& (x_6-x_8)(x_6-x_{11})(x_7-x_8)(x_7-x_9)(x_7-x_{11})
(x_8-x_9)(x_9-x_{10}).
\end{eqnarray*}
A direct computation in \textsc{Mathematica} shows that the coefficient of
\[
x_1x_2^4x_3^3x_4^2x_5^3x_6^2x_7^2x_8^3x_9^2x_{10}^2x_{11}^2
\]
is $-1$.

\item \textbf{Case $F_i=F_6$.}

The function $f$ is given by (see Figure~\ref{F1-color})
\[
f(v_i)=
\begin{cases}
2, & i\in\{4,5\},\\
3, & i\in\{9,10,11\},\\
4, & i\in\{1,6\},\\
5, & i\in\{2,3,7,8\}.
\end{cases}
\]

The graph polynomial of $F_6^2$ is
\begin{eqnarray*}
P_{F_6^2}(\bm{x})
&=& \prod_{u\sim v,\;u<v}(x_u-x_v)\\
&=& (x_1-x_2)(x_1-x_3)(x_1-x_8)(x_1-x_9)(x_1-x_{10})
(x_2-x_3)(x_2-x_4)\\
&& (x_2-x_9)(x_2-x_{10})(x_3-x_4)(x_3-x_5)(x_3-x_{10})
(x_4-x_5)\\
&& (x_4-x_6)(x_5-x_6)(x_5-x_7)(x_5-x_{11})(x_6-x_7)
(x_6-x_8)\\
&& (x_6-x_{11})(x_7-x_8)(x_7-x_9)(x_7-x_{11})(x_8-x_9)
(x_8-x_{11})(x_9-x_{10}).
\end{eqnarray*}
A direct computation in \textsc{Mathematica} shows that the coefficient of
\[
x_1^2x_2^3x_3^3x_4x_5x_6^2x_7^4x_8^4x_9^2x_{10}^2x_{11}^2
\]
is $-1$.

\item \textbf{Case $F_i=F_7$.}

The function $f$ is given by (see Figure~\ref{F7-color})
\[
f(v_i)=
\begin{cases}
2, & i\in\{4,5\},\\
3, & i\in\{9,10,11\},\\
4, & i\in\{1,8\},\\
5, & i\in\{2,3,6,7\}.
\end{cases}
\]

The graph polynomial of $F_7^2$ is
\begin{eqnarray*}
P_{F_7^2}(\bm{x})
&=& \prod_{u\sim v,\;u<v}(x_u-x_v)\\
&=& (x_1-x_2)(x_1-x_3)(x_1-x_8)(x_1-x_9)(x_1-x_{10})
(x_2-x_3)(x_2-x_4)\\
&& (x_2-x_9)(x_2-x_{10})(x_3-x_4)(x_3-x_5)(x_3-x_{10})
(x_4-x_5)\\
&& (x_4-x_6)(x_5-x_6)(x_5-x_7)(x_6-x_7)(x_6-x_8)
(x_6-x_{11})\\
&& (x_7-x_8)(x_7-x_9)(x_7-x_{11})(x_8-x_9)(x_8-x_{11})
(x_9-x_{10})(x_9-x_{11}).
\end{eqnarray*}
A direct computation in \textsc{Mathematica} shows that the coefficient of
\[
x_1x_2^3x_3^4x_4x_5x_6^4x_7^4x_8^2x_9^2x_{10}^2x_{11}^2
\]
is $-1$.

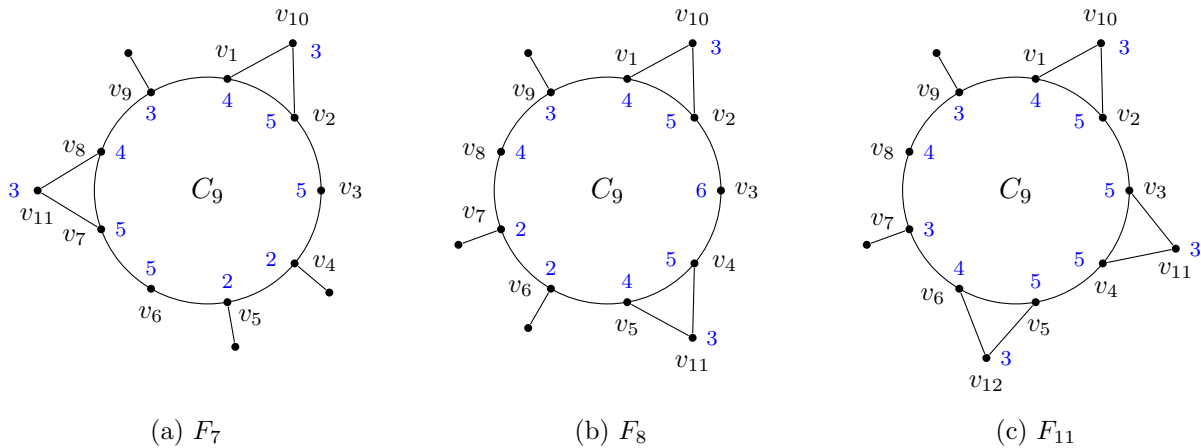
\begin{figure*}[htbp]
\begin{multicols}{3}
\begin{center}
\begin{tikzpicture}
[u/.style={fill=black, minimum size =3pt,ellipse,inner sep=1pt},node distance=1.5cm,scale=1.5]
\node[u] (v1) at (80:1){};
\node[u] (v2) at (40:1){};
\node[u] (v3) at (0:1){};
\node[u] (v4) at (320:1){};
\node[u] (v5) at (280:1){};
\node[u] (v6) at (240:1){};
\node[u] (v7) at (200:1){};
\node[u] (v8) at (160:1){};
\node[u] (v9) at (120:1){};
\node[u] (v10) at (60:1.5){};
\node[u] (v11) at (180:1.5){};
\node (C9) at (0, 0){$C_9$};

\draw   (0,0) circle[radius=1cm];
  \draw (v1) -- (v10);
  \draw (v2) -- (v10);
  \draw (v7) -- (v11);
  \draw (v8) -- (v11);

\node[u] (w4) at (320:1.4){};
\node[u] (w5) at (280:1.4){};
\node[u] (w9) at (120:1.4){};
  \draw (v4) -- (w4);
  \draw (v5) -- (w5);
 \draw (v9) -- (w9);

   \node[above=0.1cm, font=\small] at (v1) {$v_1$};
   \node[right=0.1cm,font=\small] at (v2) {$v_2$};
   \node[right=0.1cm, font=\small] at (v3) {$v_3$};
   \node[right=0.1cm,font=\small] at (v4) {$v_4$};
    \node[below=0.2cm,right=0.01cm, font=\small] at (v5) {$v_5$};
   \node[below=0.1cm,font=\small] at (v6) {$v_6$};
   \node[below=0.1cm, left=0.05cm, font=\small] at (v7) {$v_7$};
   \node[above=0.05cm, left=0.05cm, font=\small] at (v8) {$v_8$};
    \node[left=0.1cm,font=\small] at (v9) {$v_9$};
   \node[above=0.1cm,  font=\small] at (v10) {$v_{10}$};
    \node[below=0.1cm, font=\small] at (v11) {$v_{11}$};

 \node[below=0.05cm,  font=\scriptsize] at (v1) {\textcolor{blue}{4}};
\node[below=0.05cm, left=0.1cm, font=\scriptsize] at (v2) {\textcolor{blue}{5}};
\node[left=0.05cm,font=\scriptsize] at (v3) {\textcolor{blue}{5}};
\node[above=0.05cm, left=0.1cm,font=\scriptsize] at (v4) {\textcolor{blue}{2}};
\node[above=0.05cm,font=\scriptsize] at (v5) {\textcolor{blue}{2}};
\node[above=0.05cm, font=\scriptsize] at (v6) {\textcolor{blue}{5}};
\node[right=0.05cm, font=\scriptsize] at (v7) {\textcolor{blue}{5}};
\node[right=0.05cm, font=\scriptsize] at (v8) {\textcolor{blue}{4}};
\node[below=0.05cm, font=\scriptsize] at (v9) {\textcolor{blue}{3}};
\node[ below=0.1cm, right=0.1cm, font=\scriptsize] at (v10) {\textcolor{blue}{3}};
\node[left=0.1cm, font=\scriptsize] at (v11) {\textcolor{blue}{3}};
 \end{tikzpicture}
      \vfill {\small (a) $F_7$   }
\end{center}
\par
\begin{center}
\begin{tikzpicture}
[u/.style={fill=black, minimum size =3pt,ellipse,inner sep=1pt},node distance=1.5cm,scale=1.5]
\node[u] (v1) at (80:1){};
\node[u] (v2) at (40:1){};
\node[u] (v3) at (0:1){};
\node[u] (v4) at (320:1){};
\node[u] (v5) at (280:1){};
\node[u] (v6) at (240:1){};
\node[u] (v7) at (200:1){};
\node[u] (v8) at (160:1){};
\node[u] (v9) at (120:1){};
\node[u] (v10) at (60:1.5){};
\node[u] (v11) at (300:1.5){};
\node (C9) at (0, 0){$C_9$};

\draw   (0,0) circle[radius=1cm];
  \draw (v1) -- (v10);
  \draw (v2) -- (v10);
  \draw (v4) -- (v11);
  \draw (v5) -- (v11);

\node[u] (w6) at (240:1.4){};
\node[u] (w7) at (200:1.4){};
\node[u] (w9) at (120:1.4){};
  \draw (v6) -- (w6);
  \draw (v7) -- (w7);
 \draw (v9) -- (w9);

   \node[above=0.05cm, font=\small] at (v1) {$v_1$};
   \node[right=0.1cm,font=\small] at (v2) {$v_2$};
   \node[right=0.05cm, font=\small] at (v3) {$v_3$};
   \node[right=0.1cm,font=\small] at (v4) {$v_4$};
    \node[below=0.1cm,  font=\small] at (v5) {$v_5$};
   \node[left=0.1cm,font=\small] at (v6) {$v_6$};
   \node[above=0.2cm, left=0.05cm, font=\small] at (v7) {$v_7$};
   \node[left=0.05cm, font=\small] at (v8) {$v_8$};
    \node[left=0.05cm,font=\small] at (v9) {$v_9$};
   \node[above=0.1cm,  font=\small] at (v10) {$v_{10}$};
    \node[below=0.1cm, font=\small] at (v11) {$v_{11}$};

 \node[below=0.05cm,  font=\scriptsize] at (v1) {\textcolor{blue}{4}};
\node[below=0.05cm, left=0.1cm, font=\scriptsize] at (v2) {\textcolor{blue}{5}};
\node[left=0.05cm,font=\scriptsize] at (v3) {\textcolor{blue}{6}};
\node[above=0.05cm, left=0.1cm,font=\scriptsize] at (v4) {\textcolor{blue}{5}};
\node[above=0.05cm,font=\scriptsize] at (v5) {\textcolor{blue}{4}};
\node[above=0.05cm, font=\scriptsize] at (v6) {\textcolor{blue}{2}};
\node[right=0.05cm, font=\scriptsize] at (v7) {\textcolor{blue}{2}};
\node[right=0.05cm, font=\scriptsize] at (v8) {\textcolor{blue}{4}};
\node[below=0.05cm, font=\scriptsize] at (v9) {\textcolor{blue}{3}};
\node[ below=0.05cm, right=0.1cm, font=\scriptsize] at (v10) {\textcolor{blue}{3}};
\node[right=0.05cm, font=\scriptsize] at (v11) {\textcolor{blue}{3}};
 \end{tikzpicture}
      \vfill {\small (b) $F_8$}

\end{center}
\par
\begin{center}
\begin{tikzpicture}
[u/.style={fill=black, minimum size =3pt,ellipse,inner sep=1pt},node distance=1.5cm,scale=1.5]
\node[u] (v1) at (80:1){};
\node[u] (v2) at (40:1){};
\node[u] (v3) at (0:1){};
\node[u] (v4) at (320:1){};
\node[u] (v5) at (280:1){};
\node[u] (v6) at (240:1){};
\node[u] (v7) at (200:1){};
\node[u] (v8) at (160:1){};
\node[u] (v9) at (120:1){};
\node[u] (v10) at (60:1.5){};
\node[u] (v11) at (340:1.5){};
\node[u] (v12) at (260:1.5){};
\node (C9) at (0, 0){$C_9$};

\draw   (0,0) circle[radius=1cm];
  \draw (v1) -- (v10);
  \draw (v2) -- (v10);
  \draw (v3) -- (v11);
  \draw (v4) -- (v11);

  \draw (v5) -- (v12);
  \draw (v6) -- (v12);

\node[u] (w7) at (200:1.4){};
\node[u] (w9) at (120:1.4){};
  \draw (v7) -- (w7);
 \draw (v9) -- (w9);

   \node[above=0.05cm, font=\small] at (v1) {$v_1$};
   \node[right=0.05cm,font=\small] at (v2) {$v_2$};
   \node[right=0.05cm, font=\small] at (v3) {$v_3$};
   \node[right=0.05cm, below=0.1cm, font=\small] at (v4) {$v_4$};
    \node[right=0.05cm, below=0.1cm, font=\small] at (v5) {$v_5$};
   \node[below=0.1cm,left=0.05cm, font=\small] at (v6) {$v_6$};
   \node[above=0.1cm, left=0.05cm, font=\small] at (v7) {$v_7$};
   \node[left=0.05cm, font=\small] at (v8) {$v_8$};
    \node[left=0.1cm,font=\small] at (v9) {$v_9$};
   \node[above=0.1cm,  font=\small] at (v10) {$v_{10}$};
    \node[below=0.05cm, font=\small] at (v11) {$v_{11}$};
    \node[below=0.1cm,  font=\small] at (v12) {$v_{12}$};

 \node[below=0.05cm,  font=\scriptsize] at (v1) {\textcolor{blue}{4}};
\node[below=0.05cm, left=0.1cm, font=\scriptsize] at (v2) {\textcolor{blue}{5}};
\node[left=0.05cm,font=\scriptsize] at (v3) {\textcolor{blue}{5}};
\node[above=0.05cm, left=0.1cm,font=\scriptsize] at (v4) {\textcolor{blue}{5}};
\node[above=0.05cm,font=\scriptsize] at (v5) {\textcolor{blue}{5}};
\node[above=0.05cm, font=\scriptsize] at (v6) {\textcolor{blue}{4}};
\node[right=0.05cm, font=\scriptsize] at (v7) {\textcolor{blue}{3}};
\node[right=0.05cm, font=\scriptsize] at (v8) {\textcolor{blue}{4}};
\node[below=0.05cm, font=\scriptsize] at (v9) {\textcolor{blue}{3}};
\node[ below=0.05cm, right=0.1cm, font=\scriptsize] at (v10) {\textcolor{blue}{3}};
\node[right=0.05cm, font=\scriptsize] at (v11) {\textcolor{blue}{3}};
\node[right=0.05cm, font=\scriptsize] at (v12) {\textcolor{blue}{3}};
 \end{tikzpicture}
        \vfill {\small (c) $F_{11}$}
\end{center}
\end{multicols} 
\caption{The numbers indicate the values of $f(v)$ at the corresponding vertices.}
\label{F7-color}
\end{figure*}

\item \textbf{Case $F_i=F_8$.}

The function $f$ is given by (see Figure~\ref{F7-color})
\[
f(v_i)=
\begin{cases}
2, & i\in\{6,7\},\\
3, & i\in\{9,10,11\},\\
4, & i\in\{1,5,8\},\\
5, & i\in\{2,4\},\\
6, & i=3.
\end{cases}
\]

The graph polynomial of $F_8^2$ is
\begin{eqnarray*}
P_{F_8^2}(\bm{x})
&=& \prod_{u\sim v,\;u<v}(x_u-x_v)\\
&=& (x_1-x_2)(x_1-x_3)(x_1-x_8)(x_1-x_9)(x_1-x_{10})
(x_2-x_3)(x_2-x_4)\\
&& (x_2-x_9)(x_2-x_{10})(x_3-x_4)(x_3-x_5)(x_3-x_{10})
(x_3-x_{11})\\
&& (x_4-x_5)(x_4-x_6)(x_4-x_{11})(x_5-x_6)(x_5-x_7)
(x_5-x_{11})\\
&& (x_6-x_7)(x_6-x_8)(x_6-x_{11})(x_7-x_8)(x_7-x_9)
(x_8-x_9)(x_9-x_{10}).
\end{eqnarray*}
A direct computation in \textsc{Mathematica} shows that the coefficient of
\[
x_1^2x_2^4x_3^5x_4^3x_5^2x_6x_7x_8^2x_9^2x_{10}^2x_{11}^2
\]
is $2$.

\item \textbf{Case $F_i=F_{11}$.}

The function $f$ is given by (see Figure~\ref{F7-color})
\[
f(v_i)=
\begin{cases}
3, & i\in\{7,9,10,11,12\},\\
4, & i\in\{1,6,8\},\\
5, & i\in\{2,3,4,5\}.
\end{cases}
\]

The graph polynomial of $F_{11}^2$ is
\begin{eqnarray*}
P_{F_{11}^2}(\bm{x})
&=& \prod_{u\sim v,\;u<v}(x_u-x_v)\\
&=& (x_1-x_2)(x_1-x_3)(x_1-x_8)(x_1-x_9)(x_1-x_{10})
(x_2-x_3)\\
&& (x_2-x_4)(x_2-x_9)(x_2-x_{10})(x_2-x_{11})
(x_3-x_4)(x_3-x_5)\\
&& (x_3-x_{10})(x_3-x_{11})(x_4-x_5)(x_4-x_6)
(x_4-x_{11})(x_4-x_{12})\\
&& (x_5-x_6)(x_5-x_7)(x_5-x_{11})(x_5-x_{12})
(x_6-x_7)(x_6-x_8)\\
&& (x_6-x_{12})(x_7-x_8)(x_7-x_9)(x_7-x_{12})
(x_8-x_9)(x_9-x_{10}).
\end{eqnarray*}
A direct computation in \textsc{Mathematica} shows that the coefficient of
\[
x_1^2x_2^3x_3^3x_4^3x_5^3x_6^3x_7^2x_8^3x_9^2x_{10}^2x_{11}^2x_{12}^2
\]
is $1$.

\end{enumerate}

For each of the six configurations above, the indicated nonzero coefficient
implies that $F_i^2$ is Alon--Tarsi $f$-choosable, which contradicts  Claim~\ref{key-remark}.
Therefore none of the configurations
$F_1,F_4,F_6,F_7,F_8$, or $F_{11}$ can occur in $G$.
Together with Claim~24 of~\cite{KLuo25}, this completes the proof of
Claim~\ref{claim-F-one}.


\section*{Acknowledgments}

We thank Professor Xuding Zhu for bringing this problem to our attention and for his insightful comments, which inspired this work.

Seog-Jin Kim was supported by the National Research Foundation of Korea (NRF) grant funded by the Korea government (MSIT) (RS-2026-25470022). Rong Luo was supported by the Simons Foundation (Grant No.~839830).


\end{document}